\documentclass[11pt]{amsart}

\usepackage{amssymb}
\usepackage{amsmath}
\usepackage{times}
\usepackage{amsthm}
\usepackage{bbm}
\usepackage{rotating}
\usepackage{tikz}
\usetikzlibrary{shapes,fit,arrows,calc,backgrounds,positioning,matrix}
\usepackage{url}

\renewcommand{\emph}{\textsl}

\newcommand{\longiso}{\stackrel{\sim}{\longrightarrow}}
\newcommand{\iso}{\stackrel{\sim}{\to}}
\newcommand{\tto}{\longrightarrow}
\newcommand{\onto}{\twoheadrightarrow}
\newcommand{\into}{\hookrightarrow}

\newcommand{\GL}{\operatorname{GL}}
\newcommand{\SL}{\operatorname{SL}}
\newcommand{\Cl}{\operatorname{Cl}}
\newcommand{\Hom}{\operatorname{Hom}}
\newcommand{\End}{\operatorname{End}}

\newcommand{\ann}{\operatorname{ann}}
\newcommand{\PSp}{\operatorname{PSp}}

\DeclareMathOperator{\Der}{Der}
\DeclareMathOperator{\Inn}{Inn}
\DeclareMathOperator{\Mat}{Mat}
\newcommand{\Tens}{\mathsf{T}}
\newcommand{\define}{\underset{\text{\textnormal{def}}}{=}}

\newcommand{\kk}{\mathbbm{k}}
\newcommand{\Z}{\mathbb{Z}}

\newcommand{\R}{\mathbb{R}}
\newcommand{\E}{\mathbb{E}}

\newcommand{\e}{\varepsilon}
\newcommand{\x}{\mathbf{x}}
\newcommand{\osum}{\operatorname{\mathsf{orb}}}

\newcommand{\RPhi}{{\rm \Phi}}
\newcommand{\RDelta}{{\rm \Delta}}
\newcommand{\RLambda}{{\rm \Lambda}}
\newcommand{\rA}{\mathsf{A}}
\newcommand{\rB}{\mathsf{B}}
\newcommand{\rC}{\mathsf{C}}
\newcommand{\rD}{\mathsf{D}}
\newcommand{\rG}{\mathsf{G}}
\newcommand{\rE}{\mathsf{E}}
\newcommand{\rF}{\mathsf{F}}

\newcommand{\We}{\mathcal{W}}
\newcommand{\cD}{\mathcal{D}}

\newcommand{\cG}{\mathcal{G}}
\newcommand{\Sy}{\mathcal{S}}

\newcommand{\upin}{\text{\ \rotatebox{90}{$\in$}\ }}

\newcommand{\feqn}[1]{%
        \setlength\fboxrule{1pt}
	\setlength\fboxsep{3mm}
        \fbox{$\displaystyle{#1}$}
        }

\theoremstyle{plain}
\newtheorem{thm}{Theorem}[section]
\newtheorem{lem}[thm]{Lemma}
\newtheorem{cor}[thm]{Corollary}
\newtheorem{prop}[thm]{Proposition}

\theoremstyle{definition}

\theoremstyle{remark}

\begin{document}

\title{Multiplicative Invariants of Root Lattices}
\author{Jessica Hamm}
\address{Department of Mathematics, Winthrop University, 
Rockhill, SC 29730}
\email{hammj@winthrop.edu}





\begin{abstract}
We describe the multiplicative invariant algebras of the root lattices of
all irreducible root systems under the action of the Weyl group. In each case,
a finite system of fundamental invariants is determined and the class group of the invariant
algebra is calculated.
In some cases, a presentation and a Hironaka decomposition of the invariant algebra is given.
\end{abstract}

\maketitle


\section{Introduction}

The most traditional setting of invariant theory arises from a linear action of a group
$G$ on an $n$-dimensional vector space $V$ over a field $\kk$.
This action can be extended to the 
symmetric algebra $S(V)$; a choice of basis for $V$ yields an explicit
isomorphism $S(V) \cong \kk[x_1,\ldots,x_n]$.
This type of action is commonly called a \emph{linear action}; the resulting
algebra of invariants $R=S(V)^G$ is often referred to as an algebra of \emph{polynomial
invariants}.
The ring theoretic properties of polynomial invariants have been thoroughly explored, 
especially for \emph{finite groups} $G$ to which we will restrict ourselves in this
paper. 
We will focus on a different branch of
invariant theory known as \emph{multiplicative invariant theory}. This theory has
emerged relatively recently and has only been studied systematically during the
past $30$ years, beginning with the work of D. Farkas in the 80's \cite{F1, F2}.
Prior to Farkas, multiplicative invariants,
also known as ``exponential invariants" or ``monomial invariants", made only a few isolated appearances in
the literature,
notably in the work of Bourbaki \cite{Bou} and Steinberg \cite{Stei}.

Multiplicative invariants arise from an action of a group $G$ on 
a lattice $L \cong \Z^n$, that is, from an
integral representation $G \to \GL(L) \cong \GL_n(\Z)$. Any such action can be uniquely
extended to a $G$-action on the group algebra
$\kk[L] \cong \kk[x_{1}^{\pm{1}},x_{2}^{\pm{1}},...,x_{n}^{\pm{1}}]$ over any
(commutative) base ring $\kk$.
Within the Laurent polynomial algebra $\kk[x_{1}^{\pm{1}},x_{2}^{\pm{1}},...,x_{n}^{\pm{1}}]$,
the lattice $L$ becomes the multiplicative group of units that is
generated by the ``variables" $x_i$ and their inverses; so $L$ is the group of
monomials in the Laurent polynomial algebra. The action of $G$ stabilizes $L$ and hence
maps monomials to monomials; this explains the terms ``multiplicative" or ``monomial"
actions. Despite some obvious formal similarities in the basic setup,
multiplicative invariant theory and its linear counterpart exhibit many strikingly different features. 
For one, other than for linear actions, 
the calculation of the multiplicative invariant algebra $\kk[L]^G$ of an arbitrary group $G$
can always be reduced to a suitable \emph{finite} quotient of $G$.
In particular, multiplicative invariant algebras $\kk[L]^G$ are always 
affine $\kk$-algebras; the determination of an explicit finite set of algebra generators, also called 
\emph{fundamental invariants}, is one of the primary goals of multiplicative invariant theory.
For details on the foregoing and further background on multiplicative invariant theory 
we refer the reader to the monograph \cite{Lo3}.

As mentioned above, the investigation of multiplicative invariants can be reduced to 
the case of finite groups. 
By a classical theorem of Jordan \cite{J}, for each given $n$, there
are only finitely many finite subgroups $G \subseteq \GL_n(\Z)$ up to conjugacy, and hence there
are only finitely many possible multiplicative invariant algebras 
$\kk[x_{1}^{\pm{1}},x_{2}^{\pm{1}},...,x_{n}^{\pm{1}}]^G$
up to isomorphism. For $n=2$, there are $12$ conjugacy classes of finite subgroups
$1 \neq G \subseteq \GL_2(\Z)$ and the corresponding multiplicative invariant
algebras have been determined in \cite{Lo3}. While it would certainly be 
desirable to give a similar description of all multiplicative invariant algebras
for other small values of $n$, the practical realization of such a database is a rather 
daunting task for now.
This is partly due to the fact that the number of conjugacy classes of
finite subgroups of $\GL_n(\Z)$ increases rather sharply with $n$. Furthermore, the 
algorithmic side to multiplicative invariant theory is still in the early stages and there 
are many challenges to overcome.
Other than for linear actions on polynomial algebras,
the degree of Laurent polynomials is not generally preserved under multiplicative actions, and hence
algorithms for the computation of 
multiplicative invariants cannot be based on a notion of degree. Moreover,
multiplicative invariants should ideally be investigated over the base ring $\kk = \Z$ rather than a field.
A recent paper of Kemper \cite{gKxx} addresses these issues and offers a first step in 
``computational arithmetic invariant theory'', but the general algorithms developed in that paper 
need to be further optimized, and perhaps specifically tailored for multiplicative invariants,
before the envisioned database
in ranks $3$ and higher becomes feasible.

Instead of aiming for an exhaustive treatment in low ranks, we have chosen 
to contribute to this database by calculating
the multiplicative invariants for certain especially important
lattices of arbitrarily large rank. Specifically we work with the so-called 
\emph{root lattices} arising in Lie theory. 
Each root system naturally leads to two lattices
on which the Weyl group of the root system acts, the root lattice and the weight lattice. 
The structure of multiplicative
invariant algebras of weight lattices under the Weyl group action is known by a theorem of
Bourbaki: they are always
polynomial algebras \cite[Th{\'e}or{\`e}me VI.3.1]{Bou}. Root lattices, on the other hand,
generally have more complicated multiplicative invariants. Since Weyl groups 
are reflection groups, it is known that multiplicative invariant algebras of root lattices are affine 
normal monoid algebras \cite[Theorem 6.1.1]{Lo3}, but the structure of the monoid in question
is a priori unclear and needs to be determined. This is the starting point of our investigations in this article.
In doing so, we may assume that the underlying root system is \emph{irreducible}, because 
the multiplicative invariant algebra of the root lattice of an arbitrary root system $\RPhi$
is just the tensor product of the multiplicative invariant algebras of the root lattices of the 
irreducible components of $\RPhi$; see Section~\ref{SS:irreducible}. 
Recall that all irreducible root systems 
have been classified and must be one of the following types: 
the four classical series, 
$\rA_n$ ($n\ge 1$), $\rB_n$ ($n\ge 2$), $\rC_n$ ($n\ge 3$), and $\rD_n$ ($n\ge 4$) or 
the five exceptional types, 
$\rE_6, \rE_7, \rE_8, \rF_4,$ and $\rG_2$ \cite[Th{\'e}or{\`e}me VI.4.3]{Bou}. 
The root lattices of $\rG_2, \rF_4,$ and  $\rE_8$ are 
all identical to their weight lattices; see Bourbaki \cite[Planches VII, VIII, IX]{Bou}. 
Thus, by Bourbaki's Theorem, 
we know that the multiplicative invariant algebras are isomorphic to polynomial algebras. 
This leaves only the exceptional types $\rE_6$ and $\rE_7$ to consider.
As for the classical root systems, the invariants for root lattices of type 
$\rA_n$ and $\rB_n$ have been calculated in \cite{Lo3}, but we amplify these results by 
explicitly giving primary invariants and some computational methods related to $\rA_n$ in this paper. 
For completeness, we include the earlier calculations along with our contributions, followed by the 
calculations for the two remaining classical types, $\rC_n$ and $\rD_n$, and for the exceptional types
$\rE_6$ and $\rE_7$,
all of which is new. In each case, we have computed a system of fundamental invariants and have 
determined some interesting features of the multiplicative invariant algebra such 
as its class group and, in some cases, a presentation of the algebra.

This paper is a summary of my dissertation which was written under 
the guidance of Martin Lorenz. The complete thesis is available 
electronically at \url{https://math.temple.edu/graduate/recentphds.html}. 
The reader may wish to refer to the thesis for additional details and background material.

\section{Preliminaries}
\label{S:prelim}

\subsection{Root Systems and Weyl Groups}

In this section, we will define root lattices and discuss what is known about 
multiplicative invariants of such lattices.  For background references on root systems 
see Bourbaki \cite{Bou} and Humphreys \cite{jH72}. We follow notation of Bourbaki throughout.

\subsubsection{Lattices Associated to a Root System}
\label{SS:lattices}

Let $\RPhi$ be a root system in Euclidean space $\E \cong \R^n$. 
For $v, w \in \E$, $w \neq 0$, put
\begin{equation*}
\langle v,w\rangle :=\frac{2 (v,w)}{(w,w)}
\end{equation*}
where $(\,.\,,\,.\,)$ is the inner product of $\E$. 

The \emph{root lattice} of $\RPhi$ is defined by
\[
L = L(\RPhi) \define \Z\RPhi  \subseteq \E 
\]
where $\Z\RPhi = \sum_{\alpha \in \RPhi} \Z\alpha$. 
If $\RDelta  = \{ \alpha_1,\dots,\alpha_n\}$ is a fixed base of $\RPhi$, then 
\[
L
= \bigoplus_{i=1}^n \Z\alpha_i \cong \Z^n
\]
The Weyl group $\We = \We(\RPhi)$ acts faithfully on $L$. 

The root lattice  $L$ is contained in 
the so-called \emph{weight lattice} of $\RPhi$, which is defined by
\[
\begin{split}
\RLambda = \RLambda(\RPhi) &\define 
\{ v \in \E \mid \langle v, \alpha \rangle \in \Z \text{ for all } \alpha \in \RPhi \} \\
&= \{ v \in \E \mid \langle v, \alpha \rangle \in \Z \text{ for all } \alpha \in \RDelta \}
\end{split}
\]
For the last equality above, see \cite[p.~67]{jH72}. Since $\We$ preserves the bracket
$\langle \,.\,,\,.\,\rangle$, it follows that
$\We$ stabilizes $\Lambda$ as well.
Under the $\R$-linear isomorphism
\begin{align*} 
\hspace{1.2in} 
&\ \E & &\longiso &  &\ \R^n\notag
\hspace{1.2in}		\\
&\upin &&& &\upin  \\ 
&\ v  & &\longmapsto & (\langle v&,\alpha_i \rangle)_1^n\notag
\end{align*}
the weight lattice $\RLambda \subseteq \E$ corresponds to $\Z^n \subseteq \R^n$.
The preimages $\varpi_i \in \RLambda$ of the standard $\Z$-basis vectors of 
$\e_i \in \Z^n$ are called the \emph{fundamental weights} with respect to $\RDelta$\,; 
they  form an $\R$-basis of $\E$\,. Thus
\begin{equation*}
\langle \varpi_i,\alpha_j \rangle = \delta_{i,j} \qquad \text{ and } \qquad
\RLambda = \bigoplus_{i=1}^n \Z\varpi_i \cong \Z^n
\end{equation*}

\subsubsection{Multiplicative Invariants: Reduction to Irreducible Root Systems}
\label{SS:irreducible}

Let $\RPhi$ be a root system and let $\We = \We(\RPhi)$ be the associated Weyl group. 
As was mentioned earlier, the multiplicative invariant algebra $\Z[\RLambda]^{\We}$
of the $\We$-action on the weight lattice $\RLambda = \RLambda(\RPhi)$ is a polynomial
algebra over $\Z$. A convenient set of variables is provided by the orbit sums
\[
\osum(\varpi_i) \define \sum_{w \in \We/\We_{\varpi_i}} \x^{w(\varpi_i)}
\]
of the fundamental weights $\varpi_i$\,. See
\cite[Th{\'e}or{\`e}me VI.3.1]{Bou} or \cite[Theorem 3.6.1]{Lo3} for a proof of this result.
Thus, in the following, we will concentrate on the multiplicative invariant algebra $\Z[L]^{\We}$ 
of the root lattice $L = L(\RPhi)$. Our goal in this section is to justify the claim made in the Introduction that
it suffices to consider the case of an irreducible root system $\RPhi$.

A root system $\RPhi$ is called \emph{irreducible} if it is not possible to write
$\RPhi$ as a disjoint union $\RPhi = \RPhi_1 \sqcup \RPhi_2$ with nonempty $\RPhi_1$
and $\RPhi_2$ that are elementwise orthogonal to each other. A general root system $\RPhi$
uniquely decomposes as a disjoint union $\RPhi = \bigsqcup_{i=1}^r \RPhi_i$ of irreducible
root systems $\RPhi_i$ that are elementwise orthogonal to each other; these are called the
irreducible components of $\RPhi$. The Weyl group
$\We = \We(\RPhi)$ is then the direct product of the Weyl groups $\We_i = \We(\RPhi_i)$, with
$\We_i$ acting trivially on all $\RPhi_j$ with $j \neq i$. See
\cite[Section VI.1.2]{Bou} for all this. It follows that $L = \bigoplus_{i=1}^r L_i$ with
$L_i = L(\RPhi_i) = \Z\RPhi_i$ and so $\Z[L] \cong \Z[L_1] \otimes_\Z \Z[L_2] \otimes \dots
\otimes_\Z \Z[L_r]$. This description of $\Z[L]$ and the above description of $\We$ easily
imply that
\[
\Z[L]^\We  \cong \Z[L_1]^{\We_1} \otimes_\Z \Z[L_2]^{\We_2} \otimes_\Z \dots
\otimes_\Z \Z[L_r]^{\We_r}
\]
Therefore, it suffices to describe the factors $\Z[L_i]^{\We_i}$, and so we may assume
that $\RPhi$ is irreducible.

\subsubsection{Multiplicative Invariants: Monoid Algebra Structure}
\label{SS:monoid}

Again, let $\RPhi$ be a root system and let $\We = \We(\RPhi)$ be the associated Weyl group. 
As was remarked in the previous section, the multiplicative invariant algebra $\Z[\RLambda]^{\We}$
of the weight lattice $\RLambda = \RLambda(\RPhi)$ is a polynomial
algebra over $\Z$, with the orbit sums
$\osum(\varpi_i) = \sum_{w \in \We/\We_{\varpi_i}} \x^{w(\varpi_i)}$
of the fundamental weights $\varpi_i$ acting as variables. 
Thus, putting
\[
\RLambda_+ = \bigoplus_{i=1}^n \Z_+\varpi_i \cong \Z_+^n
\]
we can view $\Z[\RLambda]^{\We}$ as the algebra of the monoid $\RLambda_+$:
\begin{equation}
\label{E:monoid}
\Z[\RLambda]^{\We} = \Z[\osum(\varpi_1),\dots,\osum(\varpi_n)] \cong \Z[\RLambda_+]
\end{equation}

The multiplicative invariant algebra $\Z[L]^{\We}$ of the root lattice $L = L(\RPhi)$ is generally not
quite as simple, but $\Z[L]^{\We}$ is at least still a monoid algebra. Specifically, recall that
$L \subseteq \RLambda$; so we may consider the submonoid $L \cap \RLambda_+$ of $L$\,.
The following result is
a special case of \cite[Proposition 6.2.1]{Lo3}.

\begin{thm}
\label{T:monoid}
Let  $L = L(\RPhi)$ be the root lattice of a root system $\RPhi$ and
let $\We = \We(\RPhi)$ be its Weyl group. The invariant algebra $\Z[L]^{\We}$
is isomorphic to the monoid algebra of $L \cap \RLambda_+$\,. On the basis
$L \cap \RLambda_+$ of the monoid algebra $\Z[L \cap \RLambda_+]$, the
isomorphism is explicitly given by
\begin{align*} 
\hspace{1.2in} 
\Omega \colon \Z[L &\cap \RLambda_+] & &\longiso &  \Z&[L]^{\We} 
\hspace{1.2in}		\\
&\upin &&& &\upin  \\ 
\sum_{i=1}^n &z_i \varpi_i  & &\longmapsto & \prod_{i=1}^n &\osum(\varpi_i)^{z_i} 
\end{align*}
\end{thm}

It remains to describe the structure of the monoid $L \cap \RLambda_+$ for a given 
irreducible root system $\RPhi$.
A method for determining a system of fundamental invariants and a Hironaka decomposition of the
invariant algebra $\Z[L]^{\We} \cong \Z[L \cap \RLambda_+]$ will be described in 
Section~\ref{SS:Hilbert} below.

\subsubsection{Hilbert Bases of Monoids}
\label{SS:Hilbert}

We will only be concerned with finitely generated submonoids of lattices; these
will be
referred to as \emph{affine monoids} in the following. 
An affine monoid $M$ is called \emph{positive} if $M$ has no units (that is, invertible elements)
other than $0$\,. In this case, an element $0 \neq m \in M $ is called \emph{indecomposable} if
$m = a + b$ $(a,b \in M)$ implies $a = 0$ or $b = 0$. 
It is a standard fact that the collection of all indecomposable elements
of any positive affine monoid $M$ is finite and forms
the unique smallest generating set of $M$; it is called the \emph{Hilbert basis} of $M$.

Now let us focus specifically on the monoid 
\[
M = L \cap \RLambda_+
\]
that was considered in Theorem~\ref{T:monoid}. This monoid is positive affine. 
A finite Hilbert basis for  $M = L \cap \RLambda_+$ can be constructed
following the procedure from \cite[6.3.5]{Lo3}, which is based on the usual proof of 
Gordan's Lemma; see, e.g., \cite[6.1.2]{BH}. 
Consider the weight lattice $\RLambda = \bigoplus_{i=1}^n \Z \varpi_i$ as before.
Since $\RLambda/L$ is finite, we may define $z_i \in \Z_{>0}$ 
to be the order of $\varpi_i$ modulo $L$ and write
\begin{equation}
\label{E:mi}
m_i = z_i \varpi_i \in L \qquad (i = 1,\ldots,n)
\end{equation}
Since the $\varpi_i$
form an $\R$-basis of Euclidean space $\E$, we have $\E = \bigoplus_{i=1}^n \R m_i$
and $L \cap \bigoplus_{i=1}^n \R_+ m_i = L \cap \bigoplus_{i=1}^n \R_+ \varpi_i
= L \cap \RLambda_+ = M$\,. Put 
\[
K =\Bigl\{ \sum_{i=1}^n t_i m_i \in \E \mid 0\le t_i\le 1\Bigr\}
\supset K^\circ = \Bigl\{  \sum_{i=1}^n r_i m_i \in \E \mid 0\le r_i < 1 \Bigr\}
\ .
\]
Then $K \cap L \subseteq M$ and $K \cap L$ is finite, being the
intersection of a compact and a discrete subset of $\E$. We claim
that $K \cap L$ generates the monoid $M$. To see this, note that
each $m \in \bigoplus_{i=1}^n \R_+ m_i$ can be uniquely written as
$m = m' + m''$
with $m' \in \bigoplus_{i=1}^n \Z_+ m_i \subseteq L$ and $m'' \in K^\circ$.
If $m \in M$,  then the summand $m''$ belongs
to $K^\circ \cap L$\,. Since $m_1,\ldots,m_n$ also belong to $K \cap L$, the
expression $m = m' + m''$
exhibits $m$ as an element of the monoid generated by $K \cap L$, which proves
our claim. Note that $m_1,\ldots,m_n$ are indecomposable elements
of $M = L \cap \RLambda_+$\,. The preceding argument shows that all other
indecomposable elements of $M$ belong to $K^\circ \cap L$. Denoting these
additional indecomposables of $M$ 
(if any) by $m_{n+1},\ldots,m_s$ we obtain the desired Hilbert
basis $\{m_1,\ldots,m_s\}$ for $M = L \cap \RLambda_+$\,.

We list the result here in the following proposition.

\begin{prop} 
\label{P:Hilbert}
Assume the notation of Theorem~\ref{T:monoid}.
\begin{enumerate}
\item
The Hilbert basis of the monoid 
$M = L \cap \RLambda_+$ is given by the elements
$m_i$ $(i=1,\dots,n)$ defined in \eqref{E:mi} together
with the indecomposable elements of $M$ that belong to
the finite subset $K^\circ \cap L$ of $M$\,.
\item
The monoid $M$ decomposes as $M = \bigsqcup_{m \in K^\circ \cap L} m + M_0$
with $M_0 = \bigoplus_{i=1}^n \Z_+ m_i$\,.
\item
Primary invariants for the invariant algebra $\Z[L]^{\We}$ are given by the elements
\[
\mu_i \define \Omega(m_i) = \osum(\varpi_i)^{z_i} \qquad(i=1,\dots,n)
\]
where $z_i$ is the order of $\varpi_i$ modulo $L$.
A Hironaka decomposition of $\Z[L]^{\We}$ is given by
\[
\Z[L]^{\We} = \bigoplus_{m \in K^\circ \cap L} \Omega(m) \Z[\mu_1,\dots,\mu_n]
\]
\end{enumerate} 
\end{prop}

\subsection{Class Groups}
\label{SS:classgroups}

The class group $\Cl(R)$ can be defined
for an arbitrary Krull domain $R$; it measures the ``unique factorization defect'' of $R$: the
class group $\Cl(R)$ is trivial precisely if $R$ is a UFD. For the detailed definition of Krull domains and
class groups, we refer to Fossum \cite{rF73} or Bourbaki \cite{nB65}.
For our purposes, it will suffice to remark that, for commutative noetherian domains, being
a Krull domain is the same as being integrally closed. This includes all multiplicative invariant algebras
$\kk[L]^G$ over any commutative integrally closed noetherian domain $\kk$.

The class group of multiplicative invariants has been determined, for
arbitrary multiplicative invariant algebras, in \cite[Theorem 4.1.1]{Lo3}. In this section, we will describe the
class group in the special case of a Weyl group acting on a root lattice.

In order to state the result, we briefly recall some general facts about reflections on arbitrary 
lattices $L \cong \Z^n$. An automorphism $s \in \GL(L) \cong \GL_n(\Z)$ is called a \emph{reflection} 
if the endomorphism $1-s \in \End(L) \cong \Mat_n(\Z)$
has rank $1$. It is not hard to show that, in this case, $s$ must be conjugate in $\GL(L)$ to 
exactly one of the following two matrices:
\[
\left( \begin{smallmatrix}
-1\\
&1&\phantom{1}\\
&&&\ddots& \\ 
&&&&&1
\end{smallmatrix}  \right)
\qquad \text{or} \qquad
\left( \begin{smallmatrix}
0&1\\
1&0\\
&&1&\phantom{1}\\
&&&&\ddots&\\ 
&&&&&&1
\end{smallmatrix} \right)
\]
In the former case, $s$ is called a \emph{diagonalizable reflection}; this case is
characterized by the isomorphism $H^1(\langle s \rangle,L)  \cong \Z/2\Z$, while in the 
non-diagonalizable case, we have $H^1(\langle s \rangle,L)  =0$\,.
See \cite[Section 1.7.1]{Lo3} for details on the foregoing as well as for a proof of the 
next lemma.

\begin{lem}
\label{L:diag}
Let $\cG$ be a finite subgroup of $\GL(L)$ for some lattice $L$ and let $\cD$ denote the subgroup
of $\cG$ that is generated by the diagonalizable reflections in $\cG$\,. Then $\cD$ is a normal
subgroup of $\cG$ that is an elementary abelian $2$-group of rank $r$, where $r$ is the
number of diagonalizable reflections in $\cG$\,.
\end{lem}

With this, we may now describe the class group of the 
multiplicative invariant algebra of a root lattice.

\begin{thm}
\label{T:cl}
Let  $L = L(\RPhi)$ be the root lattice of a root system $\RPhi$,
let $\We = \We(\RPhi)$ be its Weyl group, and let $\cD$ denote the subgroup of $\We$ 
that is generated by the diagonalizable
reflections. Then:
\begin{enumerate}
\item
$\Cl(\Z[L]^{\We}) \cong  H^1(\We/\cD, L^\cD)$\,. 
\item
If 
$\cD=\{1\}$ then $\Cl(\Z[L]^{\We}) \cong  \Lambda/L$, where $\Lambda$ is the
weight lattice of $\RPhi$\,.
\item
If $\RPhi$ is irreducible and $\cD \neq \{ 1\}$ then $\Cl(\Z[L]^{\We}) = 0$\,.
\end{enumerate} 
\end{thm}

\subsection{Veronese Algebras}
\label{S:VA}

In this section, we collect some general ring theoretic facts on Veronese algebras. 
Much of this material is well-known or ``folklore'', at least for algebras over a field.
However, I am not aware of a reference for parts (c) and (d) of Proposition~\ref{P:Veronese} below.
Since we are working over the integers $\Z$ in this paper, full proofs
for an arbitrary commutative base ring $\kk$ are given. In fact, for our work on
multiplicative invariants of root lattices, we will only need to consider some very special Veronese
algebras as in Corollary~\ref{L:Veronese} below. However, we hope that the more general 
Proposition~\ref{P:Veronese} may be useful elsewhere.

Let $R = R^0 \oplus R^1 \oplus R^2 \oplus \dots$ be an
arbitrary graded $\kk$-algebra, not necessarily commutative.
Thus, all $R^i$ are $\kk$-submodules of $R$ and $R^iR^j \subseteq R^{i+j}$ holds for all $i$ and 
$j$. We will mostly be interested in the
case where the algebra $R$ is generated by $R^1$ and satisfies $R^0 = \kk$.
In this case, we will say that $R$ is \emph{$1$-generated}, for short.
Let $\Tens(R^1)$ denote the tensor algebra of
the $\kk$-module $R^1$; this is a graded $\kk$-algebra with $i^{\text{th}}$ homogeneous
component 
\[
(R^1)^{\otimes i} = \underbrace{R^1 \otimes R^1 \otimes
\dots \otimes R^1}_{i \text{ factors}} 
\]
where $\otimes = \otimes_\kk$ and $(R^1)^{\otimes 0} = \kk$.
The algebra $R$ is $1$-generated iff the canonical map $\Tens(R^1) \to R$, 
given by the embedding $R^1 \into R$,
is surjective. We let $I_R$ denote the kernel of this map. 
The ideal $I_R$ is called the \emph{relation ideal} 
and any collection
generators of $I_R$ is called a \emph{set of defining relations} for the algebra $R$.
Thus, we have an exact sequence
\begin{equation*}
0 \tto I_R \tto \Tens(R^1) \tto R \tto 0
\end{equation*}
Since the epimorphism $\Tens(R^1) \onto R$ maps $(R^1)^{\otimes i}$
to $R^i$,
we have $I_R = \bigoplus_{i\ge 0} I^i_R$ with $I^i_R = I_R \cap R^i$. Thus, we may always
choose \emph{homogeneous} defining relations for $R$. Thus, the above short exact
sequence amounts to short exact sequences
\begin{equation}
\label{E:reln}
0 \tto I_R^i \tto (R^1)^{\otimes i} \tto R^i \tto 0
\end{equation}
for each $i \ge 0$.
It is easy to see that $R$ has no defining relations of degree $i$ 
if and only if the following map, given by multiplication, is surjective:
\begin{equation}
\label{E:degi}
(R^1\otimes I_R^{i-1}) \oplus (I_R^{i-1} \otimes R^1) \tto I_R^i
\end{equation}

For a given positive integer $c$, the $c^{\text{th}}$ \emph{Veronese subalgebra} of $R$
is defined by 
\[
R^{(c)} = \bigoplus_{i \ge 0} R^{ci}
\]
We will view the Veronese subalgebra 
$R^{(c)}$ as a graded algebra in its own right,  with $i^{\text{th}}$ homogeneous
component 
\[
(R^{(c)})^i = R^{ci}
\]

\begin{prop}
\label{P:Veronese}
Let $R = \bigoplus_{i \ge 0} R^{i}$ be a graded $\kk$-algebra, where
$\kk$ is some commutative ring, and let
$R^{(c)} = \bigoplus_{i \ge 0} R^{ci}$ be the the $c^{\text{th}}$ Veronese subalgebra of $R$. Then:
\begin{enumerate}
\item
If $R$ is $1$-generated, then $R^{(c)}$ is also $1$-generated.
\item
Let $R$ be $1$-generated. If $R$ has no defining relations of degree $> (i-1)c + 1$,
then $R^{(c)}$ has no defining relations of degree $> i$. 
\item
If $R$ is a (commutative) Cohen-Macaulay ring, then so is $R^{(c)}$.
\item
If $R$ is a (commutative) Krull domain, then so is $R^{(c)}$. The embedding
$R^{(c)} \into R$ gives rise to a homomorphism of class groups $\Cl(R^{(c)}) \to \Cl(R)$.
\end{enumerate}
\end{prop}

\begin{proof}
Our proofs of (a) and (b) 
closely follow \cite[Section 3.2]{PP}. Throughout, let us put $S = R^{(c)}$ for brevity. 

(a) 
If $R$ is generated by $R^1$, then all multiplication maps 
$(R^1)^{\otimes i}  \to R^i$
are surjective. It follows that $R^iR^j = R^{i+j}$ holds for all $i$ and $j$, not just $\subseteq$. We
further conclude that multiplication maps 
\[
(S^1)^{\otimes i} = (R^c)^{\otimes i} = \underbrace{R^c \otimes_\kk R^c \otimes_\kk
\dots \otimes_\kk R^c}_{i \text{ factors}} \to R^{ci} = S^i
\]
are surjective, which proves (a).

(b)
In view of \eqref{E:degi} we need to show that 
\begin{equation}
\label{E:degi2}
(S^1 \otimes I_S^{j-1}) \oplus (I_S^{j-1} \otimes S^1) \tto I_S^j \quad \text{ is onto for $j > i$}
\end{equation}
Since the multiplication map $(R^1)^{\otimes c} \tto S^1 = R^c$ is onto,
\eqref{E:reln} yields the following commutative diagrams with exact rows, for each $j \ge 0$:
\begin{center}
\begin{tikzpicture}[description/.style={fill=white,inner sep=2pt}, >=latex]
\matrix (m) [matrix of math nodes, row sep=3em,
column sep=2.5em, text height=1.5ex, text depth=0.25ex]
{ 0 & I_R^{cj} & (R^1)^{\otimes cj} & R^{cj} & 0 \\
0 & I_S^{j} & (S^1)^{\otimes j} & S^{j} & 0 \\ };
\path[->,font=\scriptsize]
(m-1-1) edge (m-1-2) 
(m-1-2) edge (m-1-3) 
(m-1-3) edge (m-1-4) 
(m-1-4) edge (m-1-5)
(m-2-1) edge (m-2-2) 
(m-2-2) edge (m-2-3) 
(m-2-3) edge (m-2-4) 
(m-2-4) edge (m-2-5);
\path[->>]
(m-1-2) edge (m-2-2)
(m-1-3) edge (m-2-3);
\path
(m-1-4) edge[transform canvas={xshift=.2ex}] (m-2-4)
(m-1-4) edge[transform canvas={xshift=-.2ex}] (m-2-4);
\end{tikzpicture}
\end{center}
Therefore, in order to prove \eqref{E:degi2}, it is enough to show that\begin{equation}
\label{E:degi3}
\big((R^1)^{\otimes c} \otimes I_R^{c(j-1)}\big) \oplus \big(I_R^{c(j-1)} \otimes (R^1)^{\otimes c}\big) 
\tto I_R^{cj} \quad \text{ is onto for $j > i$}
\end{equation}
By our hypothesis on $R$ and \eqref{E:degi}, we know that
$(R^1\otimes I_R^{k-1}) \oplus (I_R^{k-1} \otimes R^1) \tto I_R^k$ 
is onto if $k -1 > c(i-1)$.
It follows that, for all positive integers $t$ with $k - t > c(i-1)$, the map
$\bigoplus_{l=0}^t \big( (R^1)^{\otimes l} \otimes I_R^{k-t} \otimes (R^1)^{\otimes (t-l)}\big)
\tto I_R^k$ is onto. Now assume that $j > i$ and let $k = cj$ and $t = 2c-1$. 
Then $k-t = c(j-2)+1 \ge c(i-1)+1$ and so the foregoing yields surjectivity of the map
\[
\bigoplus_{l=0}^{2c-1} \big( (R^1)^{\otimes l} \otimes I_R^{c(j-2)+1} \otimes (R^1)^{\otimes (2c-1-l)}\big)
\tto I_R^{cj} 
\]
Finally, the above map factors through the map in \eqref{E:degi3}, and hence the latter map is
surjective as well, which was to be shown.

(c)
From now on, we assume that $R$ is commutative.
Clearly, $R$ is integral over $S = R^{(c)}$. Also, the projection of $R \onto S$ with kernel
$\bigoplus_{c\, \nmid\, i} R^i$ is a ``Reynolds operator'', that is, the map is $S$-linear and the restriction 
to $S$ is equal to the identity on $S$.
Therefore, by a result of Hochster and Eagon \cite[Theorem 6.4.5]{BH}, the Cohen-Macaulay property 
descends from $R$ to $S$.

(d)
Assume that $R$ is a Krull domain and let $K$ denote the field of fractions of $R$.
Moreover, let $F$ denote the field of fractions of $S = R^{(c)}$; so $F \subseteq K$. In order
to show that $S$ is a Krull domain,
it suffices to prove that $S = R \cap F$; see \cite[Proposition 1.2]{rF73}. The inclusion $\subseteq$
being clear, consider an element $0 \neq a \in R \cap F$. Then there is a nonzero
$b \in S$ such that $ab \in S$. We need to show that this forces $a \in S$. Suppose otherwise.
Then we may assume that $a = a_s + (\text{ components of higher degree })$ with $0 \neq a_s \in R^s$
and $c \nmid d$. Similarly, write $b = b_{ct} + (\text{ components of higher degree })$ with 
$0 \neq b_{ct} \in R^{ct}$. Then $ab = a_sb_{ct} + (\text{ components of higher degree })$ 
with $0 \neq a_sb_{ct} \in R^{s+ct}$, contradicting the fact that $ab \in S$. This shows that
$S$ is a Krull domain. As for class groups, we have already pointed out in the proof of (c) that
$R$ is integral over $S$. The homomorphism $\Cl(S) \to \Cl(R)$ now follows from
\cite[Proposition 6.4(b)]{rF73}.
\end{proof}

It follows from part (b) above that if $R$ is $1$-generated with no defining relations of degree 
$> c + 1$, then $R^{(c)}$ has no defining relations of degree $> 2$.

\begin{cor}
\label{L:Veronese}
Let $R = \kk[t_1,t_2,\dots,t_d]$ denote the 
(commutative) polynomial algebra in $d\ge 2$ commuting variables over the 
commutative ring $\kk$ and let $S = R^{(2)}$ denote the second Veronese subring of $R$\,. Then:
\begin{enumerate}
\item
$S$ has algebra generators $x_i = t_i^2$ with $1 \le i \le d$ and $x_{i,j} = t_it_j$ with $1 \le i < j \le d$\,.
\item
Defining relations for $S$ are given by $[x_i,x_j] = [x_i,x_{k,l}] = [x_{k,l},x_{r,s}] = 0$
and $x_jx_j = x_{i,j}^2$.
\item
If $\kk$ is Cohen-Macaulay, then so is $S$. Indeed, we have the decomposition
\[
S = \bigoplus_{1 \le i_1 < j_1 < i_2 < \dots < i_r < j_r \le d} 
x_{i_1,j_1} x_{i_2,j_2}\dots x_{i_r,j_r} \, \kk[x_1,x_2,\dots,x_d]
\]
where we allow $r = 0$, the corresponding summand being $\kk[x_1,x_2,\dots,x_d]$.
\item
If $\kk$ is a Krull domain, then so is $S$. If $\kk$ has characteristic $\neq 2$, 
then $\Cl(S) \cong \Cl(\kk) \oplus \Z/2\Z$.
\end{enumerate}
\end{cor}

\begin{proof}
Parts (a) -- (c) follow more or less directly from the corresponding parts of Proposition~\ref{P:Veronese}.

For (a), note that the elements $x_i$ and $x_{i,j}$ generate the $\kk$-module $R^2 = S^1$,
and hence they form algebra generators of $S$. 

For (b), use the fact that $[t_i,t_j] = 0$ for $i<j$ are defining relations for $R$, of degree $2$. 
In view of the remark just before the statement of the corollary, $S$ has no defining relations 
of degree $> 2$. It is easy to see that the indicated relations are exactly the relations among the 
$x_i$ and $x_{i,j}$ of degree $\le 2$.

For (c), recall that the polynomial algebra $R$ is Cohen-Macaulay if (and only if) the
base ring $\kk$ is Cohen-Macaulay. Therefore, $S$ is Cohen-Macaulay as well by 
Proposition~\ref{P:Veronese}(c). For the indicated Hironaka decomposition,
note that the generators in (a) and the relations in (b) immediately imply that 
\[
S = \sum_{1 \le i_1 < j_1 < i_2 < \dots < i_r < j_r \le d} 
x_{i_1,j_1} x_{i_2,j_2}\dots x_{i_r,j_r} \, \kk[x_1,x_2,\dots,x_d]
\]
Since the variables $t_i$ are algebraically independent, this sum is direct.

Finally, for (d), assume that $\kk$ is a Krull domain. Then the polynomial algebra
$R = \kk[t_1,t_2,\dots,t_d]$ is a Krull domain as well by \cite[Proposition 1.6]{rF73}, and hence
so is $S$ by Proposition~\ref{P:Veronese}(d).

As for the structure of the class group, we use the fact that $S$ is free over $\kk$; in fact all
homogeneous components $R^i$ of the polynomial algebra
$R = \kk[t_1,t_2,\dots,t_d]$ are free over $\kk$. Therefore, the embedding $\kk \into S$ gives rise
to a map $\Cl(\kk) \to \Cl(S)$ by \cite[Proposition 6.4(a)]{rF73}. By Proposition~\ref{P:Veronese}(d)
we also have a map $\Cl(S) \to \Cl(R)$ coming from the inclusion $S \into R$. The composite
map $\Cl(\kk) \to \Cl(S) \to \Cl(R)$ is an isomorphism $\Cl(\kk) \cong \Cl(R)$ by \cite[Theorem 8.1]{rF73}.
It follows that $\Cl(\kk)$ injects as a direct summand into $\Cl(S)$. The image of $\Cl(\kk)$ in $\Cl(S)$ is
generated by the classes of all primes of the form $\mathfrak{p}S$, where $\mathfrak{p}$ is a 
height-$1$ prime of $\kk$; it is easy to see that $\mathfrak{p}S$ is indeed a prime of $S$ (of height $1$).
On the other hand, by \cite[Corollary 7.2]{rF73}, there is a surjection $\Cl(S) \onto \Cl(S_{\mathcal{C}})$,
where $\mathcal{C}$ denotes the set of nonzero elements of $\kk$, and the kernel of this map is generated 
by the very same primes $\mathfrak{p}S$. Therefore, $\Cl(S)/\Cl(\kk) \cong \Cl(S_{\mathcal{C}})$. Finally,
$S_{\mathcal{C}}$ is just the second Veronese algebra of the polynomial algebra 
$K[t_1,t_2,\dots,t_d]$, where $K$ is the field of fractions of $\kk$. The class group of $S_{\mathcal{C}}$
has been determined in \cite[Example 1 on page 58]{Samuel} to be isomorphic to
$\Z/2\Z$. This proves part (d).
\end{proof}

Since we will be exclusively concerned with commutative algebras in later
sections, we will not list the commuting relations, such as the relations
$[x_i,x_j] = [x_i,x_{k,l}] = [x_{k,l},x_{r,s}] = 0$ in Corollary~\ref{L:Veronese},
in our future results. Thus, in the context of commutative algebras,
the defining relations for the Veronese algebra
$S = R^{(2)}$ in Corollary~\ref{L:Veronese} are the relations $x_jx_j = x_{i,j}^2$.

\section{Type $\rB_n$}
\label{SS:Bn}

We start with the root lattice for the root system of type $\rB_n$ $(n \ge 2)$, because it has particularly nice 
invariants and will be used as an 
aid in calculating the invariants of the other three classical root lattices. We are mainly
interested in the multiplicative invariants under the Weyl group 
$\We = \We(\rB_n)$; however we also 
include the invariants under the symmetric group $\Sy_n \le \We$ here, since they 
will be useful for finding invariants of the root lattice of type $\rA_n$. To find the invariants for $\rB_n$ we 
need not invoke the general methods described above in \ref{SS:monoid} and \ref{SS:Hilbert};
instead we use a more straightforward direct approach.
In fact, all that is needed here is the fundamental theorem for 
$\Sy_n$-invariants; see, e.g., \cite[Th{\'e}or{\`e}me 1 on p.~A IV.58]{BouAlg}.

\subsection{Root system, root lattice and Weyl group}
\label{SSS:rootBn}

The root system of type $\rB_n$ is the following subset of $\E = \R^n$:
\begin{equation}
\label{E:BnPhi}
\RPhi = \{ \pm \e_i \mid 1 \le i \le n \} \cup \{ \pm \e_i \pm \e_j \mid 1 \le i < j \le n \}
\end{equation}
The root lattice $L(\RPhi) = \Z\RPhi$ will be denoted by $B_n$; so
\[
B_n = \bigoplus_{i=1}^n \Z \e_i
\]
The Weyl group $\We = \We(\RPhi)$ is the semidirect product of the group of all permutation
matrices in $\GL(\E) = \GL_n(\R)$ with the group of all
diagonal matrices $\cD_n \le \GL_n(\Z)$. The group of permutation matrices is isomorphic
to the symmetric group $\Sy_n$,
operating by permuting the basis $\{\e_i\}_1^n$ via $\sigma(\e_i) = \e_{\sigma(i)}$. The diagonal group
$\cD_n \cong \{ \pm 1\}^n$ operates via $\e_i \mapsto \pm \e_i$. Thus,
\[
\We =\cD_n \rtimes \Sy_n \cong \{ \pm 1\}^n \rtimes \Sy_n
\]

\subsection{Diagonalizable reflections}
\label{SSS:diagBn}

We determine the subgroup $\cD \le \We$ that is
generated by the diagonalizable reflections on $B_n$\,; see Lemma~\ref{L:diag} and Theorem~\ref{T:cl}.
Note that $\cD_n$ is generated by the 
diagonalizable reflections $d_i$ with $d_i(\e_j) = \e_j$ for $i \neq j$ and $d_i(\e_i) = -\e_i$.
Thus, we have 
$\cD_n \le \cD$. In fact, equality holds here. To see this, recall from Lemma~\ref{L:diag} that
$\cD$ is abelian; so $\cD \cap \Sy_n$ is contained in the centralizer $C_{\Sy_n}(\cD_n)$ of
$\cD_n$ in $\Sy_n$\,. Since $C_{\Sy_n}(\cD_n) = \{ 1\}$, we must have
\[
\cD = \cD_n
\]

\subsection{Multiplicative $\We$-invariants}
\label{SSS:multBn}

Setting $x_i:= \x^{\e_i}$ we form the group algebra, 
\[
\Z[B_n] = \Z[x_1^{\pm 1}, x_2^{\pm 1}, \dots, x_n^{\pm 1}] 
\]
We start by determining invariants under the normal subgroup $\cD_n$. This has been
carried out in detail in \cite[Example 3.5.1]{Lo3}, but we briefly review the calculation.
Since  $\Z[B_n] \cong \Z[x^{\pm 1}]^{\otimes n}$
and $\cD_n \cong \{ \pm 1\}^n$, it is easy to see that
$\Z[B_n]^{\cD_n} \cong \left( \Z[x^{\pm 1}]^{\{ \pm 1\}}\right)^{\otimes n}$.
It is also straightforward to check that $\Z[x^{\pm 1}]^{\{ \pm 1\}} = \Z[x + x^{-1}]$.
Thus, we obtain
\[
\Z[B_n]^{\cD_n} = \Z[\varphi_1, \dots, \varphi_n] \qquad \text{with}\quad
\varphi_i :=x_i + x_i^{-1}
\]
This is a polynomial algebra over $\Z$\,.
The subgroup $\Sy_n \le \We$ permutes the variables $\varphi_i$ in the standard fashion: 
$\sigma(\varphi_i) = \varphi_{\sigma(i)}$\,. Since $\Z[B_n]^{\We} =  \Z[\varphi_1, \dots, \varphi_n]^{\Sy_n}$,
the fundamental theorem for 
$\Sy_n$-invariants yields the final result:
\begin{equation}
\label{E:Bn}
\feqn{\Z[B_n]^{\We} = \Z[\sigma_1, \dots, \sigma_n]}
\end{equation}
where $\sigma_i$ denotes the $i^{\text{th}}$ elementary symmetric function in the 
variables $\varphi_1,\dots,\varphi_n$. In particular, we see that $\Z[B_n]^{\We}$ is a 
polynomial algebra over $\Z$, giving 
\[
\Cl(\Z[B_n]^{\We})=0
\]
This is of course consistent with Theorem~\ref{T:cl}(c).
We also mention that the fact that $\Z[B_n]^{\We}$ is a polynomial algebra is also a consequence of the 
Bourbaki's theorem for multiplicative invariants of weight lattices,
because $B_n$ is isomorphic to the weight lattice of the root system of type $\rC_n$
which will be discussed later.

\subsection{Multiplicative $\Sy_n$-invariants}
\label{SSS:multBnSn}

Now we restrict the group action on $B_n$ to the permutation subgroup $\Sy_n \le \We$. 
First notice that 
\[
\Z[B_n] = \Z[x_1^{\pm 1}, x_2^{\pm 1}, \dots, x_n^{\pm 1}] = \Z[x_1, x_2, \dots, x_n][s_n^{-1}]
\]
where $s_n = x_1x_2\dots x_n$ is the $n^{\text{th}}$ elementary symmetric polynomial in the 
variables $x_i$. 
Just as above, the $\Sy_n$ action is given by $\sigma(x_i)=x_{\sigma(i)}$ for all $\sigma \in S_n$. 
The fundamental theorem for $\Sy_n$-invariants gives 
$\Z[x_1, x_2, \dots, x_n]^{\Sy_n}\cong \Z[s_1, s_2, \dots, s_n]$ where $s_j$ is the 
$j^{\text{th}}$ elementary symmetric function in the variables $x_i$. 
Since we clearly have
\[
(\Z[x_1, x_2, \dots, x_n][s_n^{-1}])^{\Sy_n} = \Z[x_1, x_2, \dots, x_n]^{\Sy_n}[s_n^{-1}]
\]
it follows that,
\begin{equation}
\label{E:BnSn}
\feqn{\Z[B_n]^{\Sy_n} = \Z[s_1,s_2, \dots, s_{n-1}, s_n^{\pm 1}]\cong \Z[\Z_+^{n-1} \oplus \Z]}
\end{equation}
This is a mixed Laurent polynomial ring or, alternatively, the monoid $\Z$-algebra of the
(additive) monoid $\Z_+^{n-1} \oplus \Z$.

\section{Type $\rA_n$}
\label{SS:An}

Next, we look at the root lattice of the root system of type $\rA_n$; this root lattice 
will be denoted by $A_n$. 
Actually, we will consider the root lattice $A_{n-1}$ $(n \ge 2)$, because this fits better
with the notation of Section~\ref{SS:Bn} which we will continue to use.

\subsection{Root system, root lattice and Weyl group}
\label{SSS:rootAn}

Here, we take $\E$ to be the subspace of $\R^n$ consisting of all points whose coordinate sum
is $0$. The root system of type $\rA_{n-1}$ is given by
\begin{equation}
\label{E:AnPhi}
\RPhi =  \{  \e_i - \e_j \mid 1 \le i, j \le n, i \neq j\}
\end{equation}
Note that $\RPhi$ is contained in the root system of type $\rB_n$ 
as displayed in \eqref{E:BnPhi}. Therefore, the
root lattice $A_{n-1} = \Z\RPhi$ is contained in the root lattice $B_n$\,. The 
Weyl group $\We = \We(\rA_{n-1})$ is the subgroup $\Sy_n \le \We(\rB_n)$ permuting 
the basis $\{\e_i\}_1^n$ of $\E = \R^n$ as usual:
\[
\We = \Sy_n
\]
The vectors
\[
\alpha_i:=\e_i-\e_{i+1} \qquad (i = 1 ,\dots,n-1)
\]
form a base of the root system $\RPhi$. So
the root lattice $A_{n-1} = L(\RPhi)$ is given by
\[
A_{n-1} = \bigoplus_{i=1}^{n-1} \Z \alpha_i 
\]
Note that there is an exact sequence of $\Sy_n$-lattices, that is, an exact sequence of free
abelian groups with $\Sy_n$-equivariant maps,
\begin{equation}
\label{E:AnSeq}
0 \tto A_{n-1} \tto B_n \tto \Z \tto 0
\end{equation}
Here $\Z$ has the trivial $\Sy_n$-action and the map $B_n \to \Z$ sends $\e_i \mapsto 1$\,.

\subsection{Multiplicative $\We$-invariants}
\label{SSS:multAn}

Using the notation $x_i = \x^{\e_i}$ as for $B_n$ above, we set $y_i:=\x^{\alpha_i}=\frac{x_i}{x_{i+1}}$ to 
get the group algebra 
\[
\Z[A_{n-1}] =  \Z[y_1^{\pm 1}, y_2^{\pm 1}, \dots, y_{n-1}^{\pm 1}]
\]
From \eqref{E:AnSeq}, we see that $\Z[A_{n-1}]$  is the degree-zero component of 
the Laurent polynomial algebra
$\Z[B_n] = \Z[x_1^{\pm 1}, x_2^{\pm 1}, \dots, x_n^{\pm 1}]$, 
graded by total degree in the variables $x_i$.
Since the action of $\We = \Sy_n$ is degree-preserving, it follows that 
the multiplicative invariant algebra $\Z[A_{n-1}]^{\Sy_n}$ is the 
degree-zero component of $\Z[B_n]^{\Sy_n} = \Z[ s_1, \dots, s_{n-1}, s_n^{\pm 1}]$; see \eqref{E:BnSn}. 
Since $\deg s_i = i$\,, it is easy to see that a $\Z$-basis for the degree-zero component 
of $\Z[ s_1,s_2, \dots, s_n^{\pm 1}]$ is given by the elements 
\[
\frac{s_1^{l_1}s_2^{l_2}\cdots s_{n-1}^{l_{n-1}}}{s_n^{l_n}} \qquad 
\text{where $l_i \in \Z_+$ and $\sum_{i=1}^{n-1} il_i = nl_n$}
\]
Hence $\Z[A_{n-1}]^{\We}$ is isomorphic to the monoid $\Z$-algebra $\Z[M_{n-1}]$,
where $M_{n-1}$ is the following submonoid of $\Z_+^{n-1}$:
\begin{equation}
\label{E:An-1'}
\feqn{M_{n-1} = \Big\{ (l_1, l_2, \dots, l_{n-1}) \in \Z_+^{n-1} \mid \sum_{i=1}^{n-1} i l_i\in (n) \Big\}}
\end{equation}
The isomorphism 
is explicitly given by
\begin{equation}
\label{E:An-1"}
\feqn{%
\begin{aligned} 
 \Z[&M_{n-1}] & &\longiso &  &\quad \Z[A_{n-1}]^{\We} \\
&\upin &&& &\quad\quad \upin  \\ 
(l_1,l_2, &\dots, l_{n-1})  & &\longmapsto & &\frac{s_1^{l_1}s_2^{l_2}\cdots s_{n-1}^{l_{n-1}}}{s_n^{l_n}} 
\end{aligned}
}
\end{equation}
with $l_n = \frac{1}{n}\sum_{i=1}^{n-1} i l_i$\,.

Using Proposition~\ref{P:Hilbert}(c) above, we can easily find the $n-1$ indecomposables 
of the monoid $M_{n-1}$ that correspond 
to the primary invariants for $\Z[A_{n-1}]^{\We}$. 
It suffices to find $z_i$, the order of our fundamental weights, $\varpi_i$, modulo $A_{n-1}$. 
Then applying the isomorphism $\Omega:\Z[A_{n-1} \cap \RLambda_+] \longiso \Z[A_{n-1}]^{\We}$ 
in Theorem~\ref{T:monoid} to the indecomposable 
elements $m_i=z_i\varpi_i$ will give us the primary invariants for this lattice. 
In detail, the fundamental weights $\varpi_i$ with respect to the above base
$\{ \alpha_i\}_1^{n-1}$ of $\RPhi$ are given by (see \cite{Bou})
\[
\scriptsize
\begin{aligned}
\varpi_i &= \e_1+\cdots + \e_i -\frac{i}{n} \sum_{j=1}^n \e_j \\
&= \frac{1}{n} [(n-i)(\alpha_1 + 2\alpha_2 +\cdots + (i-1)\alpha_{i-1}) + i((n-i)\alpha_i + (n-i-i)\alpha_{i+1} +\cdots + \alpha_{n-1})]\\
&= \alpha_1 + 2\alpha_2 +\cdots + (i-1)\alpha_{i-1} -
\frac{i}{n}(\alpha_1 + 2\alpha_2 +\cdots +(i-1)\alpha_{i-1} - (n-i)\alpha_i -\cdots -\alpha_{n-1}) \\
\end{aligned}
\]
To guarantee that $m_i=z_i\varpi_i \in A_{n-1}$, we must have integral coefficients for all $\alpha_i$. So 
$z_i$ must be the smallest positive integer that satisfies the condition $z_i\cdot \frac{i}{n} \in \Z$. 
Explicitly, $z_i=\frac{n}{i_n}$ with 
\[
i_n=\gcd(i,n)
\]
Now applying $\Omega$ gives primary invariants $\pi_i=\osum(\varpi_i)^{\frac{n}{i_n}}$. 
One can easily calculate the $\Sy_n$-orbit sum of $\varpi_i$ from the first expression for 
$\varpi_i$ given above:
\[
\osum(\varpi_i)=s_is_n^{-\frac{i}{n}}
\]
where $s_i$ is the $i^{\text{th}}$ elementary symmetric function in $x_1,\dots,x_n$ as before. 
Hence $\Z[A_{n-1}]^{\We}$ has the following primary invariants:
\begin{equation}
\label{E:An-1primary}
\feqn{\pi_i = s_i^{\frac{n}{i_n}}s_n^{-\frac{i}{i_n}} \quad i=1, \dots, n-1}
\end{equation}

\subsection{Computations}
\label{SS:computationsAn}

For general $n$, it is difficult to write down all secondary invariants,  a Hironaka decomposition, 
and the defining relations for $\Z[A_{n-1}]^{\Sy_n}$. In this section,
we will use the computer algebra system CoCoA to find Hilbert bases for the monoid
$M_{n-1}$ from \eqref{E:An-1'} for some specific values of $n$. 
Once we have the Hilbert basis for $M_{n-1}$ available, we can use Theorem~\ref{T:monoid} 
and Proposition~\ref{P:Hilbert}
to find the fundamental invariants and a Hironaka decomposition of the invariant algebra $\Z[A_{n-1}]^{\We} 
\cong \Z[M_{n-1}]$.
We will use the following description of the monoid $M_{n-1}$, which is clearly equivalent
to \eqref{E:An-1'}:
\[
M_{n-1} 
= \Big\{(l_1, l_2, \dots,  l_n) \in \Z_+^n \mid \sum_{i=1}^{n-1} il_i -nl_n =0 \Big\}
\]
Thus, $M_{n-1}$ is the kernel in $\Z_+^n$ of the following matrix: 
\[
A=[1,2, 3, \cdots, n-1, -n]
\]
Here are two sample calculations with CoCoA, for $n=3$ and $n=4$: 

\vspace{.3in}

\noindent $\mathbf{n=3}$: 

\begin{minipage}[t]{2in}
\begin{verbatim}
A:=Mat([[1,2,-3]]);
HilbertBasisKer(A);
[[1, 1, 1], [3, 0, 1], [0, 3, 2]]
\end{verbatim}
\end{minipage}

\bigskip

\noindent The first two components of the output vectors tell us 
the Hilbert basis of our monoid $M_2$: $(1,1)$, $(3,0)$ and $(0,3)$. 
 By \eqref{E:An-1"}, these generators
correspond to the fundamental
invariants $\pi_1 = \frac{s_1^3}{s_3}$, $\pi_2 = \frac{s_2^3}{s_3^2}$ and  
$\mu = \frac{s_1s_2}{s_3}$ respectively. 
The monoid relation $3m_3 = m_1 + m_2$ becomes 
$\mu^3 = \pi_1\pi_2$ in $\Z[A_2]^{\Sy_3}$.
This yields the following presentation for the multiplicative invariant algebra: 
\[
\Z[A_2]^{\Sy_3} = \Z[\pi_1,\pi_2, \mu]  \cong \Z[x,y,z]/(z^3-xy)
\]
Evidently, a Hironaka decomposition of $\Z[A_2]^{\Sy_3}$ is
\[
\Z[A_2]^{\Sy_3} = \Z[\pi_1,\pi_2]  \oplus \mu \Z[\pi_1,\pi_2] \oplus \mu^2 \Z[\pi_1,\pi_2]
\]
When explicitly written out in terms of the standard generators $y_i$ of the Laurent
polynomial algebra
$\Z[A_{n-1}] =  \Z[y_1^{\pm 1}, y_2^{\pm 1}, \dots, y_{n-1}^{\pm 1}]$, 
the above fundamental invariants $\pi_1,\pi_2, \mu$ have rather
unwieldy expressions. A more economical system of 
fundamental invariants for $\Z[A_{2}]^{\Sy_3}$ is given by 
\begin{align*}
\mu-3 &= y_1 + y_1^{-1} + y_2 + y_2^{-1} + y_1y_2 + y_1^{-1}y_2^{-1} &&= \osum(\alpha_1) \\
\pi_1-3\mu +3 &= y_1^2y_2 + y_1^{-1}y_2 + y_1^{-1}y_2^{-2} &&= \osum(2\alpha_1+\alpha_2)\\
\pi_2-3\mu +3 &= y_1y_2^2 + y_1y_2^{-1} + y_1^{-2}y_2^{-1} &&= \osum(\alpha_1+2\alpha_2)
\end{align*}

\vspace{.3in}

\noindent $\mathbf{n=4}$: 

\begin{minipage}[t]{2in}
\begin{verbatim}
A:=Mat([[1,2,3,-4]]);
HilbertBasisKer(A);
[[0, 2, 0, 1], [1, 0, 1, 1], [2, 1, 0, 1], 
[0, 1, 2, 2], [4, 0, 0, 1], [0, 0, 4, 3]]
\end{verbatim}
\end{minipage}

\bigskip

\noindent Here we obtain 
the following Hilbert basis of $M_3$: $(0,2,0)$, $(1,0,1)$, $(2, 1, 0)$, $(0, 1, 2)$,
$(4, 0, 0)$ and $(0,0,4)$. To obtain a presentation of the invariant algebra 
$\Z[A_3]^{\Sy_4} \cong \Z[M_3]$,
we will consider the polynomial
algebra $\Z[x, y, z, u, v, w]$ in six variables and the 
epimorphism 
\[
\Z[x, y, z, u, v, w] \onto S:= \Z[M_3]
\]
sending each variable to one of the elements in our Hilbert basis:
\[
\begin{aligned}
x &\mapsto (4, 0, 0) &
y &\mapsto (0, 2, 0) &
z &\mapsto (0, 0, 4) \\
u &\mapsto (2, 1, 0) &
v &\mapsto (1, 0, 1) &
w &\mapsto (0, 1, 2) 
\end{aligned}
\]
We will denote the kernel of this map by $I$.
Our goal is to find generators of the ideal $I$; these are the desired defining relations for $S$. 
To this end, we view $M_3 \subseteq \Z_+^3$ and the algebra $S$ as contained in the 
polynomial algebra in three variables:
\[
S = \Z[M_3] \subseteq T := \Z[\Z_+^3] 
\cong \Z[a, b, c]
\]
For example, the Hilbert basis element $(4,0,0)$ of $M_3$ becomes the monomial $a^4$ 
when viewed in $\Z[a^{\pm 1}, b^{\pm 1}, c^{\pm 1}]$,
and $(0, 1, 2)$ becomes $bc^2$.
The Laurent polynomial algebra $\Z[a^{\pm 1}, b^{\pm 1}, c^{\pm 1}]$ can be presented as
the image of the polynomial algebra in the variables $a, b, c$ plus one extra variable, $d$, which
serves as the inverse of the product $abc$. Thus, we may consider the map
\[
T[x, y, z, u, v, w] \onto T
\]
that is the identity on $T$ and maps the variables $x,\dots, w$ as indicated above. 
Thus, we have nine variables altogether. The ideal $I$
arises as the so-called elimination ideal, eliminating the variables $a,b,c,d$.
For more details, we refer to Algorithm 4.5 in \cite[page 32]{Sturm}. 
The computation is carried out with the computer algebra
system \textsc{Magma} (V2.19-10):

\begin{verbatim}
> Z := IntegerRing();
> S<a,b,c,x,y,z,u,v,w>:= PolynomialRing(Z,9);
> I:=ideal<S|x-a^4,y-b^2,z-c^4,u-a^2*b,v-a*c,w-b*c^2>;
> EliminationIdeal(I,3);
Ideal of Polynomial ring of rank 9 over Integer Ring
Order: Lexicographical
Variables: a, b, c, x, y, z, u, v, w
Basis:
[
    z*u - v^2*w,
    y*v^2 - u*w,
    y*z - w^2,
    x*w - u*v^2,
    x*y - u^2,
    x*z - v^4
]
\end{verbatim}

To summarize, we have obtained the following presentation of our multiplicative invariant 
algebra:

{\footnotesize
\[
\begin{aligned}
\Z[A_3]^{\Sy_4} &\cong \Z[M_3] \\
&\cong \Z[x, y, z, u, v, w]/(
 zu - v^2w,
    yv^2 - uw,
    yz - w^2,
    xw - uv^2,
    xy - u^2,
    xz - v^4)
\end{aligned}
\]
}

\noindent
The first defining relation, $zu - v^2w$, for example, comes from the equation $(0,0,4) + (2,1,0) = 2(1,0,1) +
(0,1,2)$ in $M_3$.

\subsection{Class group}
\label{SSS:clAn}

For $A_1$, one can easily check that the invariant algebra is a polynomial ring giving $\Cl(\Z[A_1]^{\Sy_2}) = 0$. 
To find the class group $\Cl(\Z[A_{n-1}]^{\Sy_n})$ for $n\ge 3$, 
we use Theorem~\ref{T:cl}. First, following \cite[Example 4.2.2]{Lo3}, we will show that 
\[
\cD=\{1\}
\]
To see this, we first notice that the only elements of $\Sy_n$ that act as reflections on $A_{n-1}$
are transpositions. Moreover, if $\sigma\in \Sy_n$ is a transposition then \cite[Lemma 2.8.2]{Lo3} gives 
$H^1(\langle \sigma \rangle,A_{n-1}) = 0$, because $\sigma$ has fixed points in 
$\{1,2,\dots,n\}$ for $n \ge 3$. Therefore,
$\sigma$ does not act as a diagonalizable reflection on $A_{n-1}$\,, proving that $\cD = \{ 1\}$,
as claimed. Now Theorem~\ref{T:cl}(a)(b) gives 
\[
\Cl(\Z[A_{n-1}]^{\Sy_n}) \cong H^1(\Sy_n, A_{n-1}) \cong \Lambda/A_{n-1}
\]
where $\Lambda$ is the weight lattice of the root system of type $\rA_{n-1}$\,. 
This group is known to be cyclic of
order $n$\,. Indeed, the factor $\Lambda/A_{n-1}$ can be found in \cite[Planche I, (VIII)]{Bou}, 
and the group $H^1(\Sy_n, A_{n-1})$ is also covered by \cite[Lemma 2.8.2]{Lo3}.

\section{Type $\rC_n$}

We continue using the notation of Sections~\ref{SS:Bn} and \ref{SS:An}.

\subsection{Root system, root lattice and Weyl group}
\label{SSS:rootCn}

Here, $\E = \R^n$ and 
\begin{equation}
\label{E:CnPhi}
\RPhi = \{ \pm 2\e_i \mid 1 \le i \le n \} \cup \{ \pm \e_i \pm \e_j \mid 1 \le i < j \le n \}
\end{equation}
This root system contains the root system of type $\rA_{n-1}$ from \eqref{E:AnPhi}, 
and it is identical to the
root system of type $\rB_n$ displayed in \eqref{E:BnPhi} except for the factor 
$2$ in the first set of roots above.
The Weyl group is exactly the same as for type $\rB_n$:
\[
\We =\cD_n \rtimes \Sy_n \cong \{ \pm 1\}^n \rtimes \Sy_n
\]
If $n=2$ then the root systems of type $\rB_n$ and $\rC_n$ are isomorphic, and so we may 
assume that $n \ge 3$ below.
By the foregoing the root lattice $L = L(\RPhi)$, which will be denoted by $C_n$, is sandwiched 
between the root lattices $A_{n-1}$ and
$B_n$. Under the exact sequence \eqref{E:AnSeq}, the root lattice $C_n$ is the preimage of $2\Z$ 
in $B_n$; so $C_n$ fits into the following short exact sequence of $\We$-lattices:
\begin{equation}
\label{E:CnSeq}
0 \tto C_n \tto B_n \tto \Z/2\Z \tto 0
\end{equation}
where each basis element $\e_i$ of $B_n$ is mapped to $\bar{1} \in \Z/2\Z$\,. 
The vectors $\alpha_i:=\e_i-\e_{i+1}$ for $i=1,\dots, n-1$ together with
$\alpha_n=2\e_n$ form a base of the root system $\RPhi$, and hence these
vectors are also a $\Z$-basis of the root lattice
of $C_n$\,.

\subsection{Multiplicative $\We$-invariants}
\label{SSS:multCn}

Sequence \eqref{E:CnSeq} will help us calculate multiplicative invariants of the root lattice $C_n$\,,
just as \eqref{E:AnSeq} did with $A_{n-1}$\,. As before, we let $x_i = \x^{\e_i}$ and 
$y_i = \x^{\alpha_i}$; so $y_i=\frac{x_i}{x_{i+1}}$ for $i=1,\dots, n-1$ and 
$y_n = x_n^2$. Then
\[
\Z[C_n] = \Z[y_1^{\pm 1}, y_2^{\pm 1}, \dots, y_n^{\pm 1}]
\]
We deduce from \eqref{E:CnSeq} that $\Z[C_n]$ is the subalgebra of 
$\Z[B_n] = \Z[x_1^{\pm 1}, x_2^{\pm 1}, \dots, x_n^{\pm 1}]$ that is spanned by the
monomials of even total degree in the $x_i$s. Using this observation, we can easily find the invariants 
$\Z[C_n]^{\We}$. Indeed, note that the action of $\Sy_n$ on 
$\Z[B_n]  = \Z[x_1^{\pm 1}, x_2^{\pm 1}, \dots, x_n^{\pm 1}]$
is degree preserving and $\cD_n$ preserves at least the \textit{parity} of the degree. Thus, 
$\We =\cD_n \rtimes \Sy_n$ preserves parities of degrees as well, which allows us to conclude that 
$\Z[C_n]^{\We}$ is the even-degree component of $\Z[B_n]^{\We} = \Z[\sigma_1, \dots, \sigma_n]$\,,
where 
\[
\feqn{\sigma_i = \sum_{\substack{I \subseteq \{ 1,2,\dots,n\}\\ |I| = i}} \ \prod_{j \in I} (x_j + x_j^{-1})}
\]
is the $i^{\text{th}}$ elementary symmetric function in the variables $x_j + x_j^{-1}$ $(j = 1,2,\dots,n)$
as in \eqref{E:Bn}. Therefore, a $\Z$-basis for $\Z[C_n]^{\We}$ is
given by $\sigma_1^{l_1}\sigma_2^{l_2}\cdots \sigma_n^{l_n}$ with $\sum_{i=1}^n il_i \in 2\Z$. 
These observations prove most of part (a) of the following theorem.

\begin{thm}
\label{T:Cn}
\begin{enumerate}
\item
\textbf{Algebra structure:}
$\Z[C_n]^{\We}$ is isomorphic to the monoid algebra 
$\Z[M_n]$ with 
\[
M_n = \Big\{ (l_1,l_2, \dots, l_n) \in \Z_+^n \mid \sum_{i=1}^n il_i \equiv 0 \bmod 2 \Big\}
\]
The isomorphism is given by 
\begin{align*} 
\hspace{1.2in} 
 \Z[&M_n] & &\longiso &  \Z&[C_n]^{\We} 
\hspace{1.2in}		\\
&\upin &&& &\upin  \\ 
(l_1,l_2, &\dots, l_n)  & &\longmapsto & \sigma_1^{l_1}\sigma_2^{l_2}&\cdots \sigma_n^{l_n} 
\end{align*}
The monoid $M_n$
decomposes as $M_n \cong \Z_+^{\lfloor \frac{n}{2} \rfloor} \oplus V$ with
\[
V = \Big\{ (k_1,k_2, \dots,k_{\lceil \frac{n}{2} \rceil}) \in \Z_+^{\lceil \frac{n}{2} \rceil } 
\mid \sum_i k_i \equiv 0 \bmod 2 \Big\}
\]
Thus, $\Z[C_n]^{\We}$ is a polynomial ring in
$\lfloor \frac{n}{2} \rfloor$ variables over  the second Veronese subring
of a polynomial algebra in $\lceil \frac{n}{2} \rceil$ variables over $\Z$.
\item
\textbf{Fundamental invariants:}
The algebra $\Z[C_n]^{\We}$ is generated by the following
 $n + {\lceil \frac{n}{2} \rceil \choose 2}$  invariants: 
\[
\begin{aligned}
\pi_i  &= \begin{cases} \sigma_i &\text{for $i$ even}\\
	\sigma_i^2 & \text{for $i$ odd} \end{cases} \\
\gamma_{i,j} &= \sigma_i\sigma_j \quad (1 \le i < j \le n \text{ and  $i,j$ both odd})
\end{aligned}
\]
The $\pi_i$ are primary invariants and the $\gamma_{i,j}$ are secondary: 
$\Z[\pi_1,\dots,\pi_n]$ is a polynomial algebra over $\Z$ and $\Z[C_n]^{\We}$
is a finite module over $\Z[\pi_1,\dots,\pi_n]$.
\item
\textbf{Hironaka decomposition:}
\[
\Z[C_n]^{\We} = 
\bigoplus_{\substack{1 \le i_1 < j_1 < i_2 < \dots < i_t < j_t \le n\\ \text{\rm all odd}}} 
\gamma_{i_1,j_1} \gamma_{i_2,j_2}\dots \gamma_{i_t,j_t} \, \Z[\pi_1,\dots,\pi_n]
\]
(Here, we allow $t = 0$, the corresponding summand being $\Z[\pi_1,\dots,\pi_n]$.)
\item
\textbf{Defining relations:}
The $\lceil \frac{n}{2} \rceil \choose 2$ relations
\[
\pi_i \pi_j = \gamma_{i,j}^2 \qquad (1 \le i < j \le n \text{ and  $i,j$ both odd})
\]
are defining relations for $\Z[C_n]^{\We}$.
\end{enumerate}
\end{thm}

\begin{proof}
The isomorphism $\Z[C_n]^{\We} \cong \Z[M_n]$, with the indicated monoid $M_n$, has been 
proved in the remarks preceding the statement of the theorem. For the decomposition
$M_n \cong \Z_+^{\lfloor \frac{n}{2} \rfloor} \oplus V$, note that 
\[
\sum_{i} il_i \equiv 0 \bmod 2 \iff \sum_{i \text{ odd}} l_i \equiv 0 \bmod 2
\]
The unrestricted components in even positions form the factor
$\Z_+^{\lfloor \frac{n}{2} \rfloor}$\,, the components in odd positions the factor $V$\,.
The decomposition $M_n \cong \Z_+^{\lfloor \frac{n}{2} \rfloor} \oplus V$ leads to an
algebra isomorphism 
\[
\Z[M_n] \cong \Z[\Z_+^{\lfloor \frac{n}{2} \rfloor}] \otimes \Z[V]
\]
where the first factor is a polynomial algebra in
$\lfloor \frac{n}{2} \rfloor$ variables over $\Z$ and $\Z[V]$ is the second Veronese subring
of the polynomial algebra $\Z[\Z_+^{\lceil \frac{n}{2} \rceil}] \cong \Z[t_1,\dots,t_{\lceil \frac{n}{2} \rceil}]$
as in Corollary~\ref{L:Veronese}. Equivalently, $\Z[M_n]$ is a 
polynomial algebra in $\lfloor \frac{n}{2} \rfloor$ variables over the Veronese subring.
This proves (a).

Now for the fundamental invariants in (b).
The $\sigma_i$ with $i$ even are the variables of the polynomial factor 
$\Z[\Z_+^{\lfloor \frac{n}{2} \rfloor}]$ above.
For the fundamental invariants of the Veronese factor $\Z[V]$, it suffices to quote
Corollary~\ref{L:Veronese}(a). This proves (b). The Hironaka decomposition 
and the defining relations in (c) and (d) are
immediate from Corollary~\ref{L:Veronese} as well.
\end{proof}

We remark that the generators $\pi_i$, $\gamma_{i,j}$ of the monoid $M_n$ exhibited 
in the proof of Theorem~\ref{T:Cn} above
are clearly indecomposable elements of $M_n$\,. Therefore, they form the \emph{Hilbert basis}
of $M_n$.

\subsection{Class group}
\label{SSS:clCn}

We have already pointed out that $\Z[C_2]^{\We}$ is a polynomial algebra over $\Z$,
giving  $\Cl(\Z[C_2]^{\We})=0$. Thus, we will assume that $n\ge 3$ below. 
Recall that $\Z[C_n]^{\We}$ is a polynomial ring in
$\lfloor \frac{n}{2} \rfloor$ variables over  the second Veronese subring
of a polynomial algebra in $\lceil \frac{n}{2} \rceil$ variables over $\Z$ by Theorem~\ref{T:Cn}(a).
It is a standard fact that $\Cl(R[x]) \cong \Cl(R)$ holds for any Krull domain $R$; see \cite[Theorem 8.1]{rF73}.
Thus, it would be possible to use the structure of the class groups of Veronese algebras  
(Corollary~\ref{L:Veronese}(d)) in order to calculate $\Cl(\Z[C_n]^{\We})$.
However, we will use Theorem~\ref{T:cl} instead. 

First we must find the subgroup $\cD$ consisting of all diagonalizable reflections in $\We$\,. 

\begin{lem}
\label{L:clCn}
If $n\ge 3$, then $\cD=\{1\}$. 
\end{lem} 

\begin{proof}
Recall from Lemma~\ref{L:diag} that $\cD$ is an elementary abelian normal $2$-subgroup of 
the Weyl group $\We = \cD_n \rtimes \Sy_n$\,. We first claim that $\cD \subseteq \cD_n$\,.
Indeed, otherwise the image of $\cD$ in $\Sy_n$ would be a nontrivial elementary abelian 
normal $2$-subgroup of $\Sy_n$, which forces $n=4$ and the image to be 
contained in the Klein $4$-subgroup
of $\Sy_4$\,. However, it is easy to see that no element of the form $d\sigma \in \We$, with
$d \in \cD_4$ and $\sigma \in \Sy_4$ a product of two disjoint $2$-cycles, acts as a reflection
on $C_n$ (or, equivalently, on $B_n$). Therefore, we must have $\cD \subseteq \cD_n$ as claimed.
Recall further, from Section~\ref{SSS:diagBn}, that the only elements of $\cD_n$ that act as reflections
are the elements $d_i \in \cD_n$ with $d_i(\e_j) = \e_j$ for $i \neq j$ and $d_i(\e_i) = -\e_i$.
We will show that $H^1(\langle d_i\rangle,C_n)=0$; so none of these elements acts as a 
diagonalizable reflection on $C_n$ (even though they
do so on $B_n$). For this, we use the short exact sequence 
\eqref{E:CnSeq}. The associated long exact cohomology sequence gives an exact sequence of groups
\[
B_n^{\langle d_i\rangle} \tto \Z/2\Z \tto H^1(\langle d_i\rangle,C_n) 
\tto H^1(\langle d_i\rangle, B_n)\tto H^1(\langle d_i\rangle, \Z/2\Z)
\]
First, the map $B_n^{\langle d_i\rangle} \to \Z/2\Z$ is onto,
because any $\e_j$ $(j\neq i)$ belongs to $B_n^{\langle d_i\rangle}$ and has nontrivial image in 
$(\Z/2\Z)^{\langle d_i\rangle}  = \Z/2\Z$\,. Therefore, the above sequence becomes
\[
0 \tto H^1(\langle d_i\rangle,C_n) 
\tto H^1(\langle d_i\rangle, B_n)\tto H^1(\langle d_i\rangle, \Z/2\Z)
\]
Next, the group $H^1(\langle d_i\rangle, \Z/2\Z) \cong \Hom(\langle d_i\rangle,\Z/2\Z) \cong \Z/2\Z$
is generated by the map $h$ with $h(d_i) = \bar{1} \in \Z/2\Z$\,.
We also know that $H^1(\langle d_i\rangle,B_n) \cong \Z/2\Z$\,, 
since $d_i$ is a diagonalizable reflection on $B_n$\,. Explicitly,
\[
\begin{split}
H^1(\langle d_i\rangle,B_n) &= \Der(\langle d_i\rangle,B_n)/\Inn(\langle d_i\rangle,B_n) \\
&\cong \ann_{B_n}(d_i + 1)B_n/(d_i-1)B_n \\
&= \langle \e_i + 2\Z \e_i\rangle
\cong \Z/2\Z
\end{split}
\]
with generator the class of the derivation $\delta \in \Der(\langle d_i\rangle,B_n)$ that
is given by $\delta(d_i) = \e_i$\,.
The map $H^1(\langle d_i\rangle, B_n)\to H^1(\langle d_i\rangle, \Z/2\Z) \cong \Hom(\langle d_i\rangle,\Z/2\Z)$
sends $\delta$ to $h$; so $H^1(\langle d_i\rangle, B_n) \iso H^1(\langle d_i\rangle, \Z/2\Z)$\,.
The last exact sequence above therefore shows that $H^1(\langle d_i\rangle,C_n)=0$ as desired.
\end{proof}

Now Theorem~\ref{T:cl}(a) gives 
$\Cl(\Z[C_n]^{\We}) \cong H^1(\We, C_n) \cong \Lambda/C_n$\,,
where $\Lambda$ is the weight lattice of the root system of type $\rC_n$\,. 
The factor $\Lambda/C_n$ is known to be isomorphic to $\Z/2\Z$; see \cite[Planche III, (VIII)]{Bou}. 

\section{Type $\rD_n$}

The last classical root lattice is type $\rD_n$ with $n \ge 4$. Continuing with 
the same notation as used above, we will 
calculate the multiplicative invariants and their class group for the root lattice $D_n$
associated to this root system.

\subsection{Root system, root lattice and Weyl group}
\label{SSS:rootDn}

The set of roots is the subset of $\E = \R^n$ given by
\begin{equation}
\label{E:DnPhi}
\RPhi = \{ \pm \e_i\pm \e_j \mid 1 \le i <j \le n \} 
\end{equation}
This root system is contained in the root system of type $\rC_n$. The 
Weyl group for $\rD_n$ is a proper subgroup of $\We(\rC_n) = \cD_n \rtimes \Sy_n$:
\[
\We = \We(\rD_n) =\left( \cD_n \cap \SL_n(\Z) \right) \rtimes \Sy_n
\]
We remark that this root system is often considered for $n \ge 3$; see \cite[Planche IV]{Bou}.
However, for $n=3$, the root system is isomorphic to the root system of type $\rA_3$. The description
of $\We$ above becomes the standard description of $\We(\rA_3) = \Sy_4$ as the 
semidirect product of $\Sy_3$ with a Klein $4$-group. Therefore, the
material below is only new for $n \ge 4$.

The vectors $\alpha_i := \e_i-\e_{i+1}$ for $i=1,\dots, n-1$ and 
$\alpha_n=\e_{n-1} + \e_n$ form a base for the root system $\RPhi$, giving in particular
a $\Z$-basis for the root lattice $D_n$. Note that the 
$\alpha_i$ with $1 \le i \le n-1$ were also part of the $\Z$-basis of the root lattice $C_n$ considered in
Section~\ref{SSS:rootCn}. The last basis vectors, $\e_{n-1} + \e_n$ for $D_n$ and 
$2\e_n$ for $C_n$, are related by $2\e_n + \alpha_{n-1} = \e_{n-1} + \e_n$\,. 
Therefore, the root lattices $C_n$ and $D_n$ are identical.
Thus,  by \eqref{E:CnSeq} above, 
\begin{equation}
\label{E:DnRootLattice}
D_n = \Bigl\{ {\textstyle\sum_{i}} z_i\e_i \in \bigoplus_{i=1}^n \Z\e_i \mid
{\textstyle\sum_{i}} z_i \text{ is even} \Bigr\}
\end{equation}

\subsection{Multiplicative $\We$-invariants}
\label{SSS:multDn}

To calculate the multiplicative invariants for the root lattice $D_n$\,, we will use 
Theorem~\ref{T:monoid} and Proposition~\ref{P:Hilbert}.
The invariant algebra $\Z[D_n]^{\We}$ could also be calculated using a more 
elementary approach, similar to what we did for $C_n$ and $A_n$\,, 
but the algebraic structure as a monoid algebra would be difficult to obtain in this way. 
Recall that Theorem ~\ref{T:monoid} states that $\Z[D_n]^{\We}$ is isomorphic to the monoid algebra 
of the monoid $D_n \cap \RLambda_+$ with $\RLambda_+ = \bigoplus_{i=1}^n \Z_+\varpi_i$\,, 
the isomorphism $\Omega:\Z[D_n \cap \RLambda_+] \longiso \Z[D_n]^{\We}$ being given by
\begin{align*} 
\hspace{1.2in} 
\Z[D_n &\cap \RLambda_+] & &\longiso &  \Z&[D_n]^{\We} 
\hspace{1.2in}		\\
&\upin &&& &\upin  \\ 
\sum_{i=1}^n &l_i \varpi_i  & &\longmapsto & \prod_{i=1}^n &\osum(\varpi_i)^{l_i} 
\end{align*}
In the following theorem, we determine $D_n \cap \RLambda_+$ explicitly and use the isomorphism above to give the description of our invariant algebra.

As usual, we put $x_i = \x^{\e_i}$ and we let 
\[
\feqn{\sigma_i = \sum_{\substack{I \subseteq \{ 1,2,\dots,n\}\\ |I| = i}} \ \prod_{j \in I} (x_j + x_j^{-1})}
\]
denote the $i^{\text{th}}$ elementary symmetric function in the variables 
$x_1 + x_1^{-1},\dots,x_n + x_n^{-1}$ $(j = 1,2,\dots,n)$;
these elements are invariant under $\We(\rB_n)$, which contains $\We$\,.
As we will review in the proof of Theorem~\ref{T:Dn} below,
the weight lattice $\RLambda$ is not contained in $\bigoplus_{i=1}^n \Z\e_i$ but in
the larger sublattice $\bigoplus_{i=1}^n \Z\tfrac{1}{2}\e_i$ of the Euclidean space $\E$; 
so we will work in this setting as well.
The Weyl group $\We$ stabilizes both lattices.
We put $y_i = \x^{\tfrac{1}{2}\e_i}$ and
\[
\feqn{%
\tau_\pm = \sum_{\substack{(d_1,\dots,d_n) \in \{ \pm 1\}^n\\ \prod_i d_i = \pm 1
}} y_1^{d_1}y_2^{d_2}\cdots y_n^{d_n} 
}
\]
Note that $\tau_+$ is the $\We$-orbit sum of the lattice element 
$\tfrac{1}{2}(\e_1+\e_2+ \cdots +\e_n)$, because the $\We$-orbit of this element
consists of all $\tfrac{1}{2}(d_1\e_1+\cdots + d_n\e_n)$ with $d_i = \pm 1$ and $\prod_i d_i = 1$\,.
Similarly, $\tau_-$ is the $\We$-orbit sum of $\tfrac{1}{2}(\e_1+\e_2+ \cdots +\e_{n-1}-\e_n)$\,. Using this along with the notation in \ref{SS:monoid} and \ref{SS:Hilbert} we give the following theorem.

\begin{thm}
\label{T:Dn}
\begin{enumerate}
\item
\textbf{Monoid algebra structure:}
$\Z[D_n]^{\We}$ is isomorphic to the monoid algebra 
$\Z[M_n]$ with  
\[
M_n = \Big\{ (l_i) \in \Z_+^n \mid l_{n-1} + l_n \in 2\Z  \quad\text{and}\quad
\tfrac{l_{n-1} + l_{n}}{2}\,n + l_{n-1} +  \sum_{\substack{i \le n-2 \\ i \text { \rm odd}}} l_i  \in 2\Z \Big\}
\]
The isomorphism is given by
\begin{align*} 
\hspace{1.2in} 
 \Z[&M_n] & &\longiso &  \Z&[D_n]^{\We} 
\hspace{1.2in}		\\
&\upin &&& &\upin  \\ 
(l_1,l_2, &\dots, l_n)  & &\longmapsto & 
\sigma_1^{l_1}\cdots\, &\sigma_{n-2}^{l_{n-2}}\tau_-^{l_{n-1}}\tau_+^{l_n}
\end{align*}
The monoid $M_n$
decomposes as $M_n \cong \Z_+^{\lfloor \frac{n-2}{2} \rfloor} \oplus W$ with
\[
W = \Big\{ (k_i ) \in \Z_+^{\lceil \frac{n+2}{2} \rceil } 
\mid k_{1} + k_2 \in 2\Z  \quad\text{and}\quad
\tfrac{k_1+k_2}{2}\,n  +   \sum_{i \ge 2} k_i  \in 2\Z \Big\}
\]

\item
\textbf{Fundamental invariants for $n$ even:}
Put 
\[
\pi_i  = \begin{cases}  \sigma_i^2 &\text{for $i$ odd}\\
	\sigma_i & \text{for $i$ even} \end{cases} \quad (1 \le i \le n-2)
\qquad \text{and} \qquad
\pi_{n-1}= \tau_{-}^2\,,\  \pi_{n} =  \tau_+^2
\]
Moreover, put 
\[
\begin{aligned}
\gamma_i &= \sigma_i\tau_{-}\tau_+ &&\text{for $1\le i \le n-2$, $i$ odd} \\
\gamma_{i,j} &= \sigma_i\sigma_j &&\text{for $1\le i < j \le n-2$ both odd}
\end{aligned}
\]
The above $\frac{1}{8}(n^2 + 6n)$ elements generate the invariant algebra $\Z[D_n]^{\We}$, with
the $\pi_i$ serving as primary invariants. 

\item
\textbf{Fundamental invariants for $n$ odd:}
Put 
\[
\pi_i  = \begin{cases}  \sigma_i^2 &\text{for $i$ odd}\\
	\sigma_i & \text{for $i$ even} \end{cases} \quad (1 \le i \le n-2)
\qquad \text{and} \qquad
\pi_{n-1}= \tau_{-}^4\,,\  \pi_{n} =  \tau_+^4
\]
Moreover, put 
\[
\begin{aligned}
\gamma_{i,j} &= \sigma_i\sigma_j &&\text{for $1\le i < j \le n-2$ both odd}\\
\gamma_{n-1,n} &=\tau_{-}\tau_+  \\
\gamma_{i,n-1} &= \sigma_i\tau_{-}^2 &&\text{for $1\le i\le n-2$ odd}\\
\gamma_{i,n} &= \sigma_i\tau_+^2 &&\text{for $1\le i\le n-2$ odd}\\
\end{aligned}
\]
The above $\frac{1}{8}(n^2+12n+3)$ elements generate the invariant algebra $\Z[D_n]^{\We}$, with
the $\pi_i$ serving as primary invariants.  
\end{enumerate}
\end{thm}

\begin{proof}
(a)
We need to describe the submonoid $D_n \cap \RLambda_+$ of 
$\RLambda_+ = \bigoplus_{i=1}^n \Z_+ \varpi_i$\,.
By \cite[Planche IV]{Bou} the root system of type $\rD_n$ has the following fundamental weights 
with respect to the base $\{ \alpha_i\}_1^n$ of the root system that was exhibited in
Section~\ref{SSS:rootDn}:
\begin{equation*}
\small
\begin{aligned}
\varpi_i &= \e_1+\e_2+ \cdots + \e_i \qquad (1\le i\le n-2) \\
&= \alpha_1+2\alpha_2+ \cdots + (i-1)\alpha_{i-1} + i(\alpha_i+\alpha_{i+1}+\cdots + \alpha_{n-2}) + \tfrac{1}{2}i(\alpha_{n-1}+\alpha_n) \\
\varpi_{n-1}&=\tfrac{1}{2}(\e_1+\e_2 +\cdots + \e_{n-2} +\e_{n-1}-\e_n) \\
&=\tfrac{1}{2}(\alpha_1 + 2\alpha_2 + \cdots + (n-2)\alpha_{n-2} + \tfrac{1}{2}n\alpha_{n-1} + \tfrac{1}{2}(n-2)\alpha_n )\\
\varpi_n &=\tfrac{1}{2}(\e_1+\e_2+\cdots + \e_{n-2}+\e_{n-1}+\e_n) \\
&= \tfrac{1}{2}(\alpha_1 + 2\alpha_2 + \cdots + (n-2)\alpha_{n-2} + \tfrac{1}{2}(n-2)\alpha_{n-1} + \tfrac{1}{2}n\alpha_n )
\end{aligned}
\end{equation*}
In view of \eqref{E:DnRootLattice},
an element $\sum_i l_i\varpi_i \in \RLambda_+$ $(l_i \in \Z_+)$ belongs to 
$D_n$ if and only if the following
two conditions are satisfied:
\begin{equation}
\label{E:DnCond1}
l_{n-1} + l_{n} \in 2\Z
\end{equation}
and 
\begin{equation}
\label{E:DnCond2}
\sum_{i=1}^{n-2} i l_i + \tfrac{1}{2}(n-2)l_{n-1} 
+ \tfrac{1}{2}n l_{n} \in 2\Z
\end{equation}
Indeed, \eqref{E:DnCond1} is equivalent to the condition 
$\sum_i l_i\varpi_i \in \bigoplus_{i=1}^n \Z\e_i$\,, while
\eqref{E:DnCond2} expresses the defining condition that $\sum_{i} z_i$ 
must be even in \eqref{E:DnRootLattice}.
Observe further that \eqref{E:DnCond2} can be rewritten as follows:
\begin{equation}
\label{E:DnCond3}
\tfrac{l_{n-1} + l_{n}}{2}\,n + l_{n-1} +  \sum_{\substack{i \le n-2 \\ i \text { odd}}} l_i  \in 2\Z
\end{equation}
This yields the monoid $M_n$ as well as the isomorphism 
\begin{align*} 
\hspace{1.2in} 
 &M_n & &\longiso &  D_n & \cap \RLambda_+ 
\hspace{1.2in}		\\
&\upin &&& &\upin  \\ 
(l_1,l_2, &\dots, l_n)  & &\longmapsto &  \sum_{i=1}^n &l_i\varpi_i
\end{align*}
The decomposition 
$M_n \cong \Z_+^{\lfloor \frac{n-2}{2} \rfloor} \oplus W$ is clear, because 
\eqref{E:DnCond2} imposes no condition on the 
$\lfloor \frac{n-2}{2} \rfloor$ components $l_i$ for even $i \le n-2$\,.
In the description of $W$, we have also relabeled $l_n$ as $k_1$, $l_{n-1}$ as $k_2$ etc. 

To justify the indicated isomorphism $\Z[M_n] \longiso \Z[D_n]^{\We}$, we need to determine
the $\We$-orbit sums $\osum(\varpi_{i})$\,. The $\We$-orbit of $\varpi_i = \e_1+\e_2+ \cdots + \e_i$
with $1\le i\le n-2$ consists of all possible $\pm \e_{j_1} \pm \e_{j_2} \pm  \cdots \pm \e_{j_i}$
with $j_1 < \dots < j_i$\,. Therefore, the corresponding orbit sum evaluates to 
\[
\osum(\varpi_{i}) = \sigma_i
\]
Finally, we have already pointed out before the statement of the theorem that
\[
\osum(\varpi_n) = \tau_+ \qquad \text{and} \qquad
\osum(\varpi_{n-1}) = \tau_-
\]
This completes the proof of (a).

\medskip

To obtain fundamental invariants for $\Z[D_n]^{\We}$, we determine the Hilbert basis for 
our monoid $M_n \cong D_n\cap \RLambda_+$
using the procedure described in Section~\ref{SS:Hilbert}.
As in that section, we put 
\[
m_i=z_i\varpi_i \qquad (i=1,2,\dots,n)
\]
where $z_i$ is the order of $\varpi_i$ modulo $D_n$. 
It is easy to find the $z_i$ for the fundamental weights $\varpi_i$: 
\begin{align}
\label{E:zi1}
z_i&= \begin{cases} 1 & \text{if $i$ is even} \\ 2 & \text{if $i$ is odd} \end{cases} 
\qquad (1\le i \le n-2) 
\\
\intertext{and}
\label{E:zi2}
z_{n-1} = z_n &= \begin{cases} 2 & \text{if $n$ is even} \\ 4 & \text{if $n$ is odd} \end{cases}
\end{align}
This gives the elements $m_i = z_i\varpi_i$ $(i=1,\dots,n)$ in the Hilbert basis, which by
(a) yield the fundamental invariants
\[
\pi_i = \osum(\varpi_i)^{z_i}
\]
in both (b) and (c).

As we have seen in Section~\ref{SS:Hilbert}, the remaining elements of the Hilbert basis of 
$D_n\cap \RLambda_+$ all belong to $D_n \cap K^\circ$, where
$K^\circ =\Bigl\{ \sum_{i=1}^n t_i m_i \in \E \mid 0 \le t_i < 1\Bigr\}$.
So we are looking for indecomposable elements of $M$ having the form
\[
(l_1,\dots,l_n) = (t_1z_1,\dots,t_nz_n) \in \Z_+^n \quad \text{with } 0 \le t_i < 1
\]
and such that conditions
\eqref{E:DnCond1} and \eqref{E:DnCond3} are satisfied.
Since each $l_i\in \Z_+$\,,  equations \eqref{E:zi1} and \eqref{E:zi2}
yield the following restrictions on what $l_i$ can be:
\begin{align}
\label{E:li1}
l_i& = \begin{cases} 0 & \text{if $i$ is even} \\ 0 \text{ or } 1 & \text{if $i$ is odd} 
\end{cases} \qquad (1\le i \le n-2)\\
\intertext{and}
\label{E:li2}
l_{n-1},l_n & \in \begin{cases} \{0,1,2,3\} & \text{if $n$ is odd} \\ 
\{0,1\} & \text{if $n$ is even} \end{cases}
\end{align}

First, suppose that $l_{n-1}=l_n=0$. Then \eqref{E:DnCond1} certainly holds and
condition  \eqref{E:DnCond3} becomes  
$\displaystyle \sum_{\substack{i \le n-2 \\ i \text { odd}}} l_i \in 2\Z$. 
Since the $l_i$ in this sum are either $0$ or $1$ by \eqref{E:li1}, 
condition  \eqref{E:DnCond3} just says that there must be 
an even number of  $l_i = 1$ for odd  $i\le n-2$\,.
The corresponding indecomposable elements of $M_n$ are obtained by taking just two
of these $l_i = 1$\,. In sum, we have obtained the following indecomposable elements
of $D_n\cap \RLambda_+$:
\[
m_{i,j} = \varpi_i + \varpi_j  \qquad (1 \le i < j \le n-2 \text{ and } i,j \text{ both odd})
\]
The element $m_{i,j}$ yields the fundamental invariants 
$\gamma_{i,j} = \sigma_i\sigma_j = \osum(\varpi_i)\osum(\varpi_i)$ in (b) and (c).
Note that
\begin{equation}
\label{E:DnReln1}
\gamma_{i,j}^2 = \pi_i\pi_{j}
\end{equation}

From now on, we assume that $l_{n-1},l_n$ are not both zero. For this we start with 

\bigskip

\emph{Case 1: $n$ is even}.
Since $l_{n-1}+l_n \in 2\Z$ by \eqref{E:DnCond1}, condition 
\eqref{E:li2} says we must have $l_{n-1}=l_n=1$. 
Now \eqref{E:DnCond3} becomes
\[
n+1 + \displaystyle \sum_{\substack{i \le n-2 \\ i \text { odd}}} l_i \in 2\Z 
\]
with all $l_i \in \{0,1\}$ by \eqref{E:li1}.
To satisfy the above condition we must have an odd number of $l_i = 1$. 
Taking exactly one nonzero $l_i=1$ gives the remaining indecomposable elements for $M_n$: 
\[
b_i= \varpi_i + \varpi_{n-1}+\varpi_n
\]
Indeed, any $n$-tuple $(l_1,\dots, l_{n-2},1,1) \in M_n$ with $2k+1$ of the $l_i = 1$ 
may be written as a sum $(l_1,\dots, l_{n-2},1,1) = (l'_1,\dots, l'_{n-2},0, 0)
+ (l_1,\dots, l_{n-2}, 1,1)$ with $2k$ of the $l'_i = 1$ and exactly one $l_j= 1$ $(j \neq i)$. 
Since the second summand corresponds to $b_j$ and the first one can be written in terms of the 
Hilbert basis elements $m_{r,s}$ constructed earlier, we have found the complete Hilbert basis of
$M_n$. By isomorphism in (a), the basis element $b_i$ yields the fundamental invariant 
$\gamma_i=\sigma_i\tau_{-}\tau_+$ 
giving a total of $\frac{1}{8}(n^2 + 6n)$  basis elements for our monoid, and hence
fundamental invariants, when $n$ is even. Note that
\begin{equation}
\label{E:DnReln2}
\gamma_i^2 = \pi_i\pi_{n-1}\pi_n
\end{equation}
Together with \eqref{E:DnReln1}, this relation shows that the invariant algebra 
$\Z[D_n]^{\We}$ is integral over the subalgebra $\Z[\pi_1,\dots,\pi_n]$; so the 
$\pi_i$ form a set of primary invariants.

\bigskip

\emph{Case 2: $n$ is odd}.
By \eqref{E:li2}, if either $l_{n-1}$ or $l_n$ are zero, the other must be $2$ to satisfy 
\eqref{E:DnCond1}. This leads to the possibilities  $(l_{n-1},l_n)=(2,0)$ or $(l_{n-1},l_n)=(0,2)$.
In either case \eqref{E:DnCond3} becomes
\[
n + 2 + \displaystyle \sum_{\substack{i \le n-2 \\ i \text { odd}}} l_i \in 2\Z
\]
with all $l_i \in \{0,1\}$ by \eqref{E:li1}.
Therefore, we must have an odd number of $l_i =1$. As above, indecomposable
monoid elements are obtained by taking one $l_i=1$ and all others zero. So we have 
indecomposable elements 
\[
m_{i,n-1}=\varpi_i + 2\varpi_{n-1}
\qquad \text{and} \qquad 
m_{i,n}=\varpi_i + 2\varpi_{n}
\]
giving fundamental invariants $\gamma_{i,n-1}=\sigma_i\tau_{-}^2$ and $\gamma_{i,n}=\sigma_i\tau_+^2$. 
Now assume that neither $l_{n-1}$ nor $l_n$ are zero; so $l_{n-1}$ and $l_n$ belong to $\{1,2,3\}$
by \eqref{E:li2} and their sum must be even by \eqref{E:DnCond1}.
First assume that $l_{n-1}=l_n$. Then \eqref{E:DnCond3} becomes
\[
l_n(n+1)+ \displaystyle \sum_{\substack{i \le n-2 \\ i \text { odd}}} l_i \in 2\Z
\]
Since $n+1$ is even, this amounts to the sum
$\sum l_i$ being even. The only indecomposable element of $M_n$ resulting from this situation
is obtained by letting all $l_i =0$ for $1\le i\le n-2$ and $\l_{n-1} = l_n = 1$, which gives the Hilbert
basis element 
\[
m_{n-1,n}=\varpi_{n-1}+\varpi_n
\]
and the corresponding fundamental invariant, $\gamma_{n-1,n}=\tau_{-}\tau_+$. 

Finally, if we allow one of $l_{n-1},l_n$ to be $1$ and the other $3$, 
then \eqref{E:DnCond3} gives one of the two equations:
\[
\begin{aligned}
2n+1 + \sum_{\substack{i \le n-2 \\ i \text { odd}}} l_i \in 2\Z \\ 
2n+3 + \sum_{\substack{i \le n-2 \\ i \text { odd}}} l_i \in 2\Z 
\end{aligned}
\]
In either case, we must have an odd number of nonzero $l_i=1$ in the sum
to satisfy this condition. It follows that we can write this as $m+m'$ where $m$ has an odd 
number of $l_i=1$ and $l_{n-1}=2$, (or $l_n=2$, depending on which we took to be $3$ above), 
and $m'$ is the indecomposable element $(0,\dots, 0, 1, 1)$. 
This completes our Hilbert basis when $n$ is odd, giving a total of $\frac{1}{8}(n^2+12n+3)$ 
indecomposable elements. Again, we see that the squares of the fundamental 
invariants $\gamma_{i,j}, \gamma_{i,n-1}$ and $\gamma_{i,n}$ with $i \le n-2$ as well as the
fourth power of $\gamma_{n-1,n}$ belong to the subalgebra $\Z[\pi_1,\dots,\pi_n]$ of
$\Z[D_n]^{\We}$; so the $\pi_i$ form a system of primary invariants.
This completes the proof.
\end{proof}

We remark that one can also use the description in Proposition ~\ref{P:Hilbert} to write 
down a Hironaka decomposition for $\Z[D_n]^{\We}$, though it is not particularly nice or 
compact so we omit it here.

\subsection{Class Group}
\label{SSS:clDn}

We will use Theorem ~\ref{T:cl} to calculate $\Cl(\Z[D_n]^{\We})$ for $n \ge 4$. To this end,
we must first find $\cD$, the group of diagonalizable reflections in $\We$.
But, as we noted earlier, $D_n$ and $C_n$ are the same lattice and $\We(\rD_n)$ is a subgroup
of $\We(\rC_n)$. Since there are no nonidentity elements of $\We(\rC_n)$ that act as a 
diagonalizable reflection on $D_n = C_n$ by 
Lemma~\ref{L:clCn} , we conclude that $\cD=\{1\}$ for $D_n$ as well.
Now Theorem~\ref{T:cl}(a)(b) gives 
\[
\Cl(\Z[D_n]^{\We}) \cong H^1(\We, C_n) \cong \Lambda/D_n
\]
where $\Lambda$ is the weight lattice of the root system of type $\rD_n$\,. 
The factor $\Lambda/D_n$ is known to be isomorphic to $\Z/2\Z\times \Z/2\Z$ for $n$ even and 
$\Z/4\Z$ for $n$ odd; see \cite[Planche III, (VIII)]{Bou}.

\section{Type $\rE_6$}

Turning now to exceptional root systems, we start with the multiplicative invariant
algebras of the root lattice $\rE_6$. 
Because the Weyl groups in question are very large, we will use 
Theorem~\ref{T:monoid} and Proposition~\ref{P:Hilbert} to calculate the multiplicative invariant algebras
for $\rE_6$ and $\rE_7$.
Throughout, we follow the notations of \cite[Planches V and VI]{Bou}.

\label{SSS:E6}

\subsection{Root system, root lattice and Weyl group}

Here, $\E$ is the subspace of $\R^8$ that is orthogonal to $\e_6 - \e_7$ and to
$\e_7 + \e_8$. The set of roots is the subset of $\E$ that is given by
\begin{multline}
\label{E:E6Phi}
\RPhi=\{ \pm \e_i \pm \e_j \mid 1\le i < j \le 5\} 
\\
\cup \ \{\pm \tfrac{1}{2} (\e_8-\e_7-\e_6 + \sum_{i=1}^5 (-1)^{\nu(i)}\e_i) \mid \sum_{i=1}^5 \nu(i) \in 2\Z\}
\end{multline}
Here $\nu(i)\in \{0,1\}$. Thus, there are $4\cdot \binom{5}{2} + 2^5 = 72$ roots. A base for this root system is given by
\[
\begin{aligned}
\alpha_1 &= \tfrac{1}{2}(\e_1+\e_8)-\tfrac{1}{2}(\e_2+\e_3+\e_4+\e_5+\e_6+\e_7) \\
\alpha_2 &= \e_1+\e_2\\
\alpha_i &= \e_{i-1}-\e_{i-2} \qquad (i=3,\dots,6)\\
\end{aligned}
\]
The root lattice of $\rE_6$ will be denoted by $E_6$; so $E_6=\bigoplus_{i=1}^6 \Z \alpha_i$.

The Weyl group $\We = \We(\rE_6)$ has order $2^7\,3^4\,5 = 51\,840$. The group $\We$ has a
number of interesting realizations; see Bourbaki \cite[Exercise 2 on page 228]{Bou} and
Humphreys \cite[Section 2.12]{Hum}.
For example, $\We$ can be described as the automorphism group of the famous configuration
of $27$ lines on a cubic surface. The rotation subgroup $\We^+ = \{ w \in \We \mid \det w = 1\}$
is isomorphic to the projective symplectic group $\PSp_4(3)$ 
over $\mathbb{F}_3$; the latter group is the unique simple group of order $25\,920$.
See the \emph{Atlas of Finite Groups} \cite{Atlas}.
Denoting the standard inner product of $\R^8$ by $(\,.\,,\,.\,)$ as usual, the quadratic form
$\frac{1}{2}(x,x)$ yields a non-degenerate quadratic form on the $6$-dimensional 
$\mathbb{F}_2$-vector space $\overline{E}_6 = E_6/2E_6$. The action of $\We$ on $\overline{E}_6$
preserves this form, and this induces an isomorphism
\begin{equation}
\label{E:WE6}
\feqn{\We \cong O_6(2)}
\end{equation}
in the notation of \cite{Atlas}.

\subsection{Multiplicative $\We$-invariants}

By Theorem~\ref{T:monoid} we know that 
\[
\Z[E_6]^{\We}\cong \Z[\RLambda_+ \cap E_6]
\]
where $\RLambda_+ = \bigoplus_{i=1}^6 \Z_+\varpi_i$ and $\varpi_i$ are the fundamental 
weights with respect to the above base $\{ \alpha_i \}_1^6$ of the root system.
Explicitly,
\[
\begin{aligned}
\varpi_1&= \tfrac{2}{3}(\e_8-\e_7-\e_6)
= \tfrac{1}{3}(4\alpha_1+ 3\alpha_2 + 5\alpha_3+6\alpha_4+4\alpha_5 + 2\alpha_6)\\
\varpi_2&= \tfrac{1}{2}(\e_1+\e_2+\e_3+\e_4+\e_5-\e_6-\e_7+\e_8)\\
&= \alpha_1+2\alpha_2+2\alpha_3+3\alpha_4+2\alpha_5+\alpha_6 \\
\varpi_3&= \tfrac{5}{6}(\e_8-\e_7-\e_6) + \tfrac{1}{2}(-\e_1+\e_2+\e_3+\e_4+\e_5) \\
&=\tfrac{1}{3}(5\alpha_1 + 6\alpha_2+10\alpha_3+12\alpha_4+8\alpha_5+4\alpha_6) \\
\varpi_4&= \e_3+\e_4+\e_5-\e_6-\e_7+\e_8\\
&= 2\alpha_1 + 3\alpha_2+4\alpha_3+6\alpha_4+4\alpha_5+2\alpha_6\\
\varpi_5&= \tfrac{2}{3}(\e_8-\e_7-\e_6) + \e_4+\e_5 \\
&= \tfrac{1}{3}(4\alpha_1+6\alpha_2+8\alpha_3+12\alpha_4+10\alpha_5+5\alpha_6)\\
\varpi_6&= \tfrac{1}{3}(\e_8-\e_7-\e_6) +\e_5 \\
&= \tfrac{1}{3}(2\alpha_1 + 3\alpha_2+4\alpha_3+6\alpha_4+5\alpha_5+4\alpha_6)
\end{aligned}
\]
In order to further describe the invariant algebra $\Z[E_6]^{\We}$, we need to analyze
the monoid
\[
M = \RLambda_+ \cap E_6
\]
\bigskip

\noindent
\textbf{Hilbert basis of $M$}.
Let $z\in M$ and write $z=l_1\varpi_1 +\cdots + l_6\varpi_6$ with $l_i\in \Z_+$.
Then we know that the coefficients of each $\alpha_i$ in our base must be integral. 
This gives the following conditions:
\[
\begin{aligned}
\text{coefficient of } &\alpha_1:\qquad & &\tfrac{4}{3}l_1+ l_2+\tfrac{5}{3}l_3+2l_4+\tfrac{4}{3}l_5+\tfrac{2}{3} l_6 \in \Z\\
&\alpha_2:& &l_1+2l_2+2l_3+3l_4+2l_5+l_6\in \Z\\
&\alpha_3:& &\tfrac{5}{3}l_1+ 2l_2+\tfrac{10}{3}l_3+4l_4+\tfrac{8}{3}l_5+\tfrac{4}{3} l_6 \in \Z\\
&\alpha_4:& &2l_1+ 3l_2+4l_3+6l_4+4l_5+2 l_6 \in \Z\\
&\alpha_5: & &\tfrac{4}{3}l_1+ 2l_2+\tfrac{8}{3}l_3+4l_4+\tfrac{10}{3}l_5+\tfrac{5}{3} l_6 \in \Z\\
&\alpha_6: & &\tfrac{2}{3}l_1+ l_2+\tfrac{4}{3}l_3+2l_4+\tfrac{5}{3}l_5+\tfrac{4}{3} l_6 \in \Z\\
\end{aligned}
\]
Note that the conditions at $\alpha_2$ and $\alpha_4$ are automatically satisfied 
for any $(l_1,\dots, l_6)\in \Z_+^6$, and the remaining conditions
do not impose any restrictions on $l_2$ and $l_4$. 
Therefore, our system of conditions can be rewritten as follows:
\[
\begin{aligned}
&4l_1+5l_3+4l_5+2l_6 \in 3\Z \hspace{.5in} & &5l_1+10l_3+8l_5+4l_6 \in 3\Z \\
&4l_1+8l_3+10l_5+5l_6 \in 3\Z & &2l_1+4l_3+5l_5+4l_6 \in 3\Z 
\end{aligned}
\]
Finally, reducing mod 3 we see that all four conditions are equivalent to the single condition
$l_1+2l_3+l_5+2l_6\in 3\Z$. Thus,
we are left with the following description of our monoid:
\[
M = \RLambda_+ \cap E_6 \cong \{(l_1,\dots, l_6)\in \Z_+^6 \mid l_1+2l_3+l_5+2l_6\in 3\Z\}
\cong \Z_+^2 \oplus W
\]
where 
\[
W=\{(k_1,k_2,k_3,k_4)\in \Z_+^4 \mid k_1+2k_2+k_3+2k_4\in 3\Z\}
\]
A Hilbert basis for the monoid $W$ can be obtained by using using the method described above in 
Proposition~\ref{P:Hilbert} or by 
simply using CoCoA as in Section~\ref{SS:computationsAn}. To carry out the latter approach,
we will use the following description of the monoid $W$:
\[
W 
= \Big\{(k_1,k_2,k_3,k_4, x) \in \Z_+^5 \mid k_1+2k_2+k_3+2k_4 - 3x =0 \Big\}
\]
Thus, $W$ is the kernel in $\Z_+^5$ of the matrix $A=[1,2,1,2, -3]$, which can be obtained 
with CoCoA as follows: 

\bigskip

\begin{verbatim}
A:=Mat([[1,2,1,2,-3]]);
HilbertBasisKer(A);

[[0, 0, 1, 1, 1], [1, 0, 0, 1, 1], [0, 1, 1, 0, 1], 
[1, 1, 0, 0, 1], [0, 0, 3, 0, 1], [1, 0, 2, 0, 1], 
[2, 0, 1, 0, 1], [3, 0, 0, 0, 1], [0, 0, 0, 3, 2], 
[0, 1, 0, 2, 2], [0, 2, 0, 1, 2], [0, 3, 0, 0, 2]]
\end{verbatim}

\bigskip

\noindent
Deleting the auxiliary fifth coordinate and reordering the above CoCoA output, we see that
the Hilbert basis of $W$ is given by the following elements:
\begin{equation}
\label{E:HilbE6.1}
\begin{aligned}
m_1&= (3,0,0,0) & 
m_2&= (0,3,0,0) &
m_3&= (0,0,3,0) &
m_4&= (0,0,0,3) \\
m_5&= (1,1,0,0) &
m_6&= (1,0,0,1) &
m_7&= (0,1,1,0) &
m_{8} &= (0,0,1,1) \\
m_{9} &= (1,0,2,0) &
m_{10} &= (0,1,0,2) &
m_{11} &= (0,2,0,1) &
m_{12} &= (2,0,1,0)
\end{aligned}
\end{equation}
In order to obtain the Hilbert basis for $M$, we need to convert each of the above $4$-tuples
$(k_1,\dots, k_4)\in \Z_+^4$ into the a $6$-tuple $(k_1,0,k_2,0,k_3,k_4)$, and we also need to 
add the $6$-tuples 
\begin{equation}
\label{E:HilbE6.2}
(0,1,0,0,0,0) \qquad \text{and} \qquad (0,0,0,1,0,0)
\end{equation}
to the list.

\bigskip

\noindent
\textbf{Fundamental invariants}.
By Theorem~\ref{T:monoid}, the isomorphism $\Z[M] \iso \Z[E_6]^{\We}$ is given by 
$(l_1,\dots, l_6) \mapsto  \prod_{i=1}^6 \osum(\varpi_i)^{l_i}$. In view of the preceding
paragraph, 
\[
\Z[M] \cong \Z[W] \otimes \Z[t_1,t_2]
\]
where the variables $t_1$ and $t_2$
correspond to the Hilbert basis elements in \eqref{E:HilbE6.2}. These give the 
following two fundamental invariants that are algebraically independent from each other and all 
other fundamental invariants:
\begin{equation*}
\osum(\varpi_2) \qquad \text{and} \qquad \osum(\varpi_4)
\end{equation*}
The remaining Hilbert basis elements $(k_1,k_2,k_3,k_4)$ from \eqref{E:HilbE6.1} give 12
additional fundamental invariants
\[
\osum(\varpi_1)^{k_1}\osum(\varpi_3)^{k_2}\osum(\varpi_5)^{k_3}\osum(\varpi_6)^{k_4}
\]
In view of the size of the Weyl group $\We$, it is not feasible or useful to  explicitly write out these orbit sums.

\bigskip

\noindent
\textbf{Relations}.
Relations between the 12 fundamental invariants 
coming from \eqref{E:HilbE6.1}, and hence a presentation
of the invariant algebra $\Z[E_6]^{\We}$, can be obtained by the method of
Section~\ref{SS:computationsAn} for $n=4$. In detail, we need 12 variables, one for each of the 
Hilbert basis elements $m_i$, and 4 extra variables for the embedding $M \into \Z_+^4$.
Here is the \textsc{Magma} computation:

\begin{verbatim}
> Z := IntegerRing();
> S<a,b,c,d,x1,x2,x3,x4,x5,x6,x7,x8,x9,x10,x11,x12>:= 
      PolynomialRing(Z,16);
> I:=ideal<S|x1-a^3,x2-b^3,x3-c^3,x4-d^3,x5-a*b,x6-a*d,
      x7-b*c,x8-c*d,x9-a*c^2,x10-b*d^2,x11-b^2*d,x12-a^2*c>;
> EliminationIdeal(I,4);
Ideal of Polynomial ring of rank 16 over Integer Ring
Order: Lexicographical
Variables: a, b, c, d, x1, x2, x3, x4, x5, x6, x7, x8, x9, 
      x10, x11, x12
Basis:
[
    x7*x10 - x8*x11,
    x6*x9 - x8*x12,
    x5*x10 - x6*x11,
    x5*x9 - x7*x12,
    x5*x8 - x6*x7,
    x6*x7*x8 - x9*x10,
    x6*x7^2 - x9*x11,
    x6^2*x7 - x10*x12,
    x5*x6*x7 - x11*x12,
    x4*x12 - x6^2*x8,
    x4*x11 - x10^2,
    x4*x9 - x6*x8^2,
    x4*x7 - x8*x10,
    x4*x5 - x6*x10,
    x3*x12 - x9^2,
    x3*x11 - x7^2*x8,
    x3*x10 - x7*x8^2,
    x3*x6 - x8*x9,
    x3*x5 - x7*x9,
    x3*x4 - x8^3,
    x2*x12 - x5^2*x7,
    x2*x10 - x11^2,
    x2*x9 - x5*x7^2,
    x2*x8 - x7*x11,
    x2*x6 - x5*x11,
    x2*x4 - x10*x11,
    x2*x3 - x7^3,
    x1*x11 - x5^2*x6,
    x1*x10 - x5*x6^2,
    x1*x9 - x12^2,
    x1*x8 - x6*x12,
    x1*x7 - x5*x12,
    x1*x4 - x6^3,
    x1*x3 - x9*x12,
    x1*x2 - x5^3
]
\end{verbatim}

\subsection{Class Group}

Using the description above along with Theorem~\ref{T:cl}, we can find the class group for the invariant algebra
$\Z[E_6]^{\We}$. We know that the Weyl group $\We \cong O_6(2)$ contains a simple subgroup of index 2,
which is its unique nontrivial normal subgroup. 
It follows that the subgroup of diagonalizable reflections, $\cD$, is trivial. 
Hence, by Theorem~\ref{T:cl}(a), the class group is isomorphic to the factor 
$\RLambda/L$. 
This group is known \cite[(VIII) in Planche V]{Bou}:
\[
\feqn{\Cl(\Z[E_6]^{\We}) \cong \Z/3\Z}
\]

\section{Type $\rE_7$}
\label{SSS:E7} 

Lastly, we calculate the multiplicative invariants for the root system of type $\rE_7$.

\subsection{Root system, root lattice and Weyl group}

Here, $\E$ is the subspace of $\R^8$ that is orthogonal 
$\e_7 + \e_8$. The set of roots is 
\begin{multline}
\label{E:E7Phi}
\RPhi=\{ \pm \e_i \pm \e_j \mid 1\le i < j \le 6\} \quad \cup \quad \{\pm (\e_7-\e_8)\} \\
\cup \quad \{\tfrac{1}{2}(\e_7-\e_8+\sum_{i=1}^6 (-1)^{\nu(i)}\e_i\mid \sum_{i=1}^8 \nu(i) \text{ odd}\} 
\end{multline}
A base for this root system is given by
\[
\begin{aligned}
\alpha_1 &= \tfrac{1}{2}(\e_1+\e_8)-\tfrac{1}{2}(\e_2+\e_3+\e_4+\e_5+\e_6+\e_7) \\
\alpha_2 &= \e_1+\e_2\\
\alpha_i &= \e_{i-1}-\e_{i-2} \qquad (i=3,\dots,7)
\end{aligned}
\]
So our root lattice is $E_7=\bigoplus_{i=1}^7 \Z \alpha_i$.

The Weyl group for $\rE_7$ is the direct product of $\Z/2\Z$ and the unique simple group of order $1\,451\,520$ (which can be described as $\PSp_6(2)$).

The Weyl group $\We = \We(\rE_7)$ has order $2^{10}\,3^4\,5\,7 = 2\,903\,040$. By Bourbaki 
\cite[Exercise 3 on page 229]{Bou} (see also Humphreys \cite[Section 2.12]{Hum}), there is an 
isomorphism of groups
\[
\feqn{ \We \cong \{ \pm 1 \} \times O_7(2) }
\]
which arises similar to the earlier description of $\We(\rE_6)$: the form
$\frac{1}{2}(x,x)$ yields a non-degenerate quadratic form on the $7$-dimensional 
$\mathbb{F}_2$-vector space $E_7/2E_7$, and this form is preserved by the action of $\We$.
The restriction of the action to the rotation subgroup $\We^+ = \{ w \in \We \mid \det w = 1\}$
is an isomorphism $\We^+ \cong O_7(2)$, and the kernel of the action is $\{ \pm 1\}$.
The latter group $O_7(2)$ is the unique simple group of order $1\,451\,520$.

\subsection{Multiplicative $\We$-invariants}

We follow the outline of our treatment of type $\rE_6$. First, we find a Hilbert basis for the monoid
\[
M = \RLambda_+ \cap E_7
\]
First we list the fundamental weights with repeat to the above base $\{ \alpha_i\}$:
\[
\begin{aligned}
\varpi_1&= \e_8-\e_7 = 2\alpha_1+2\alpha_2+3\alpha_3+4\alpha_4+3\alpha_5+2\alpha_6+\alpha_7\\
\varpi_2&= \tfrac{1}{2}(\e_1+\e_2+\e_3+\e_4+\e_5+\e_6-2\e_7+2\e_8)\\
&= \tfrac{1}{2}(4\alpha_1+7\alpha_2+8\alpha_3+12\alpha_4+9\alpha_5+6\alpha_6 +3\alpha_7) \\
\varpi_3&= \tfrac{1}{2}(-\e_1+\e_2+\e_3+\e_4+\e_5+\e_6-3\e_7+3\e_8) \\
&=3\alpha_1+4\alpha_2+6\alpha_3+8\alpha_4+6\alpha_5+4\alpha_6+2\alpha_7 \\
\varpi_4&= \e_3+\e_4+\e_5+\e_6+2(\e_8-\e_7)\\
&= 4\alpha_1 + 6\alpha_2+8\alpha_3+12\alpha_4+9\alpha_5+6\alpha_6 + 3\alpha_7\\
\varpi_5&= \e_4+\e_5+\e_6+\tfrac{3}{2}(\e_8-\e_7) \\
&= \tfrac{1}{2}(6\alpha_1+9\alpha_2+12\alpha_3+18\alpha_4+15\alpha_5+10\alpha_6+ 5\alpha_7)\\
\varpi_6&= \e_5+\e_6-\e_7+\e_8\\
&= 2\alpha_1 + 3\alpha_2+4\alpha_3+6\alpha_4+5\alpha_5+4\alpha_6 +2\alpha_7\\
\varpi_7&= \e_6+\tfrac{1}{2}(\e_8-\e_7)\\
&=\tfrac{1}{2}(2\alpha_1+3\alpha_2+4\alpha_3+6\alpha_4+5\alpha_5+4\alpha_6+3\alpha_7)
\end{aligned}
\] 

\bigskip

\noindent\textbf{Hilbert basis of $M$}.
Note that $\varpi_1$, $\varpi_3$, $\varpi_4$ and $\varpi_6$ already belong to $E_7$.
Therefore,
\[
M = \Z_+ \varpi_1 \oplus \Z_+ \varpi_3 \oplus  \Z_+ \varpi_4 \oplus  \Z_+ \varpi_6 \oplus M' 
\cong \Z_+^4 \oplus M'
\]
with
\[
M' = E_7 \cap \bigoplus_{i=2,5,7} \Z_+ \varpi_i
\]
Now let $l_2\varpi_2 + l_5\varpi_5  + l_7\varpi_7 \in M'$, with $l_i\in \Z_+$. 
Then the coefficients of each $\alpha_i$ in our base must be integral. 
One can easily check that the coefficients of $\alpha_1, \alpha_3, \alpha_4$ and $\alpha_6$ 
are automatically integral. Looking at the coefficients of  the remaining elements of our base gives
\[
\begin{aligned}
\text{coefficient of } &\alpha_2:\qquad & &\tfrac{7}{2}l_2+ \tfrac{9}{2}l_5 + \tfrac{3}{2}l_7 \in \Z\\
&\alpha_5:& &\tfrac{9}{2}l_2+ \tfrac{15}{2}l_5 + \tfrac{5}{2}l_7\in \Z\\
&\alpha_7:& &\tfrac{3}{2}l_2+ \tfrac{5}{2}l_5+ \tfrac{3}{2}l_7\in \Z\\
\end{aligned}
\]
This reduces to the single condition $l_2+l_5+l_7\in 2\Z$, giving 
\[
M' \cong \{(l_2,l_5, l_7) \in \Z_+^3 \mid l_2+l_5+l_7\in 2\Z\}
\]
As above, one easily finds the following Hilbert basis for the monoid $M'$:
\[\footnotesize
\begin{aligned}
m_1&= (2,0,0) \\
m_2&= (0,2,0) \\
m_3&= (0,0,2) \\
m_4&= (1,1,0) \\
m_5&= (1,0,1) \\
m_{6} &= (0,1,1) \\
\end{aligned}
\]

\bigskip

\noindent\textbf{Structure of the invariant algebra $\Z[E_7]^{\We}$}.
Note that the monoid algebra $\Z[M']$ is just the second Veronese subalgebra $R^{(2)}$ of a polynomial algebra 
$R = \Z[t_2,t_5,t_7]$ in
three variables. The structure of such algebras has been explained in Corollary~\ref{L:Veronese}.
It follows that the invariant algebra $\Z[E_7]^{\We}$ has the following description:
\[
\Z[E_7]^{\We} \cong \Z[M] \cong \Z[M'] \otimes \Z[t_1,t_3,t_4,t_6] \cong R^{(2)}[t_1,t_3,t_4,t_6] 
\] 
a polynomial algebra in four variables over $R^{(2)}$. Fundamental invariants are given by 
the following four that correspond to the variables $t_1,t_3,t_4,t_6$,
\[
\osum(\varpi_1),\ \osum(\varpi_3),\ \osum(\varpi_4),\  \osum(\varpi_6)
\]
together with the following six invariants that
correspond to the above monoid generators $m_1,\dots,m_6$,
\[
\osum(\varpi_2)^2, \ \osum(\varpi_5)^2, \  \osum(\varpi_7)^2, \ \osum(\varpi_2)\osum(\varpi_5), \ 
\osum(\varpi_2)\osum(\varpi_7), \ 
\osum(\varpi_5)\osum(\varpi_7)
\]

\subsection{Class Group}

From the structure of the Weyl group, $ \We \cong \{ \pm 1 \} \times O_7(2) $, we know that the only nontrivial
normal subgroups of $\We$ are $\{ \pm 1 \}$ and $O_7(2)$. Only the former is an elementary abelian $2$-group, but 
$-1$ is not a reflection. By Lemma~\ref{L:diag} it follows that the subgroup of diagonalizable reflections, $\cD$, is trivial. 
Hence, by Theorem~\ref{T:cl}(a), the class group is isomorphic to the factor 
$\RLambda/L$, which is known \cite[(VIII) in Planche VI]{Bou}:
\[
\feqn{\Cl(\Z[E_7]^{\We}) \cong \Z/2\Z}
\]
Alternatively, we could arrive at the same conclusion using the fact
that the invariant algebra $\Z[E_7]^{\We}$ is a polynomial algebra over the second Veronese subalgebra 
$R^{(2)}$ of $R = \Z[t_2,t_5,t_7]$. Indeed, by
\cite[Theorem 8.1]{rF73}, it follows that $\Cl(\Z[E_7]^{\We}) \cong \Cl(R^{(2)})$, and the latter group was calculated in 
Corollary~\ref{L:Veronese}.

\bibliographystyle{amsplain}
\bibliography{References}

\end{document}